\documentclass[reqno]{amsart}

\usepackage[sorting=none, sorting=nyt, maxbibnames=99, backend=biber]{biblatex}

\renewbibmacro{in:}{}
\bibliography{references.bib}
\usepackage{amssymb}
\usepackage{calrsfs}
\usepackage{graphics,graphicx, mathrsfs}
\usepackage{enumerate}
\usepackage{enumitem}
\usepackage{url}
\usepackage{xcolor}
\definecolor{vio}{rgb}{0.54, 0.17, 0.89}
\usepackage[ruled, lined, linesnumbered, longend]{algorithm2e}
\usepackage{hyperref}
\usepackage{titlesec}

\newtheorem{theorem}{Theorem}[section]
\newtheorem{lemma}[theorem]{Lemma}

\newtheorem{proposition}[theorem]{Proposition}
\newtheorem{conjecture}[theorem]{Conjecture}
\newtheorem{corollary}[theorem]{Corollary}
\numberwithin{equation}{section}

\theoremstyle{remark}
\newtheorem*{remark}{Remark}

\titleformat{\section}
  {\normalfont\large\bfseries\centering}{\thesection}{1em}{}

\titleformat{\subsection}
  {\normalfont\bfseries}{\thesubsection}{1em}{}  

\DeclareMathOperator{\Z}{\mathbb{Z}}

\def\reals{\hbox{\rm I\kern-.18em R}}
\def\complexes{\hbox{\rm C\kern-.43em
\vrule depth 0ex height 1.4ex width .05em\kern.41em}}
\def\field{\hbox{\rm I\kern-.18em F}} %symbol for field

\newcommand\blfootnote[1]{%
  \begingroup
  \renewcommand\thefootnote{}\footnote{#1}%
  \addtocounter{footnote}{-1}%
  \endgroup
}

\newenvironment{section*}[2][A]{
  \section*{#2}
  \renewcommand\thesection{#1}
  \setcounter{theorem}{0}}{}

\allowdisplaybreaks

\begin{document}

\title[New bounds on the summatory function of $(-2)^{\Omega(n)}$]{New bounds and progress towards a conjecture on the summatory function of $(-2)^{\Omega(n)}$}

\author{Daniel R. Johnston, Nicol Leong, Sebastian Tudzi}
\address{School of Science, UNSW Canberra, Australia}
\email{daniel.johnston@unsw.edu.au}
\address{School of Science, UNSW Canberra, Australia}
\email{nicol.leong@unsw.edu.au}
\address{School of Science, UNSW Canberra, Australia}
\email{s.tudzi@unsw.edu.au}
\date\today
% \subjclass[2000]{Primary: 11Y40; Secondary: 11M06, 11M26, 11B65, 11B68}
\keywords{}

\begin{abstract}
    In this article, we study the summatory function
    \begin{equation*}
        W(x)=\sum_{n\leq x}(-2)^{\Omega(n)},
    \end{equation*} 
    where $\Omega(n)$ counts the number of prime factors of $n$, with multiplicity. We prove $W(x)=O(x)$, and in particular, that $|W(x)|<2260x$ for all $x\geq 1$. This result provides new progress towards a conjecture of Sun, which asks whether $|W(x)|<x$ for all $x\geq 3078$. To obtain our results, we computed new explicit bounds on the Mertens function $M(x)$. These may be of independent interest. 
    
    Moreover, we obtain similar results and make further conjectures that pertain to the more general function
    \begin{equation*}
        W_a(x)=\sum_{n\leq x}(-a)^{\Omega(n)}
    \end{equation*}
    for any real $a>0$. 
\end{abstract}

\maketitle
\blfootnote{\textit{Affiliation}: School of Science, The University of New South Wales Canberra, Australia.}
\blfootnote{\textit{Corresponding author}: Daniel Johnston (daniel.johnston@unsw.edu.au).}
\blfootnote{\textit{Key phrases}: prime omega function, Mertens function, explicit number theory.}
\blfootnote{\textit{2020 Mathematics Subject Classification}: 11N56, 11M41 (Primary). 11Y70 (Secondary).}

\section{Introduction}
Let $\Omega(n)$ denote the number of prime factors of $n$, counting multiplicity. In this paper, we are concerned with bounding the function
\begin{equation}\label{weq}
    W(x)=\sum_{n\leq x}(-2)^{\Omega(n)}.
\end{equation}
In particular, we are motivated by the following conjecture of Sun \cite[Conjecture 1.1]{sun2016pair}, which has been verified by Mossinghoff and Trudgian \cite{mossinghoff2021oscillations} for $x\leq 2.5\cdot 10^{14}$.
\begin{conjecture}[Sun]\label{suncon}
    For all $x\geq 3078$, the bound $|W(x)|<x$ holds.
\end{conjecture}
For comparison, the function
\begin{equation}\label{gxeq}
    G(x)=\sum_{n\leq x}2^{\Omega(n)}
\end{equation}
satisfies $G(x)\sim Ax(\log x)^2$ for some constant $A>0$ \cite[Theorem 2]{grosswald1956average}. Sun's conjecture therefore implies a $(\log x)^2$ cancellation upon replacing the positive terms $2^{\Omega(n)}$ with oscillatory terms $(-2)^{\Omega(n)}$. Another key observation is that $W(x)$ has jump discontinuities of size $x$ whenever $x$ is a power of $2$. Hence, there exist infinitely many values of $x$ such that
\begin{equation}\label{x2eq}
    |W(x)|\geq\frac{x}{2}.
\end{equation}
Therefore, if Sun's conjecture were true, it would be very close to the best bound possible.

Sun's conjecture is also of a similar form to other historical conjectures. Let
\begin{equation}\label{mandleq}
    M(x)=\sum_{n\leq x}\mu (n)\qquad\text{and}\qquad L(x)=\sum_{n\leq x}(-1)^{\Omega(n)},
\end{equation}
where
\begin{equation*}
    \mu(n)=
    \begin{cases}
        (-1)^{\Omega(n)},&\text{if $n$ is square-free},\\
        0,&\text{otherwise}
    \end{cases}
\end{equation*} 
is the M\"obius function. In 1897, Mertens \cite{mertens1897zahlentheoretische} famously conjectured that $|M(x)|<\sqrt{x}$ for all $x>1$, and in 1919 P\'olya \cite{polya1919verschiedene} asked whether $L(x)\leq 0$ for all $x~\geq~2$. Both of these inequalities were eventually found to fail infinitely often (see respectively, \cite{odlyzko1985disproof, haselgrove1958disproof}). The first integer value of $x$ for which $L(x)>0$ is $x=906\ 150\ 257$ (see \cite{tanaka1980numerical}). However, the first value of $x$ for which $|M(x)|\geq \sqrt{x}$ is very large and has not yet been explicitly determined. Extensive computations \cite{hurst2018computations,kotnik2006mertens} have only shown that such an $x$ must satisfy $10^{16}<x<\exp(1.59\cdot 10^{40})$. In addition, $M(x)$ is believed to grow asymptotically faster than $\sqrt{x}$, with an unpublished conjecture of Gonek asserting that
\begin{equation}
    \limsup_{x\to\infty}\left|\frac{M(x)}{\sqrt{x}(\log\log\log x)^{5/4}}\right|=B
\end{equation} 
for some constant $B>0$. See \cite{ng2004distribution} for further discussion on this conjecture, and \cite{ngconjecture2024} for a conjectural value of the constant $B$.
There are also several other examples (see e.g.,\ \cite{mossinghoff2017liouville,mossinghoff2020tale,mossinghoff2021oscillations}) of functions similar to $W(x)$ with analogous behaviour. That is, functions which appear to satisfy an elegant bound, but ultimately fail to do so at large values of $x$.

Therefore, by considering the failure of such past conjectures, one may expect that Sun's conjecture is also false. In addition, one may expect that $|W(x)|$ grows at a rate asymptotically faster than $x$. However, using a convolution argument, and updating existing bounds on $M(x)$, we are able to prove the following.

\begin{theorem}\label{mainthm}
    There exists a constant $C>0$ such that for all sufficiently large $x$,
    \begin{equation*}
        |W(x)|<Cx.
    \end{equation*}
    That is, $W(x)=O(x)$. Moreover, one can take $C=2260$ for all $x\geq 1$.
\end{theorem}
We are therefore able to show that Sun's conjecture predicts the correct order of $W(x)$, but are unable to prove the conjecture in full. Interestingly, even if one considers sufficiently large $x$ (rather than all $x\geq 1$), our method yields no improvement to the value of $C$. This is because the worst case in our proof is that of large $x$ (see Section \ref{proofsect}). The most routine way to lower the value of $C$ further would be to obtain better explicit bounds for the function
\begin{equation}\label{littlemeq}
    m(x)=\sum_{n\leq x}\frac{\mu(n)}{n},
\end{equation}
which is closely related to $M(x)$. Even so, obtaining Sun's conjectured value of $C=1$ (with $x\geq 3078$) appears out of reach via our method of proof.

Combining Theorem \ref{mainthm} with \eqref{x2eq}, one has the following corollary.
\begin{corollary}\label{limsupcor}
    The limit supremum
    \begin{equation*}
        s_W=\limsup_{x\to\infty}\frac{|W(x)|}{x}
    \end{equation*}
    exists and satisfies $\frac{1}{2}\leq s_W\leq 2260$.
\end{corollary}
The authors believe that determining an explicit value for $s_W$ would be more insightful than proving Sun's conjecture, which only implies $s_W\leq 1$.

For any real $a>0$, we can similarly study the function
\begin{equation*}
    W_a(x)=\sum_{n\leq x}(-a)^{\Omega(n)}
\end{equation*}
for which $W_2(x)=W(x)$ and $W_1(x)=L(x)$. The case $0<a<2$ is more simply treated by complex analysis and has been widely studied in the literature, see e.g.,\ \cite[Section II.6]{tenenbaum2015introduction}. Therefore, we will focus on the case $a\geq 2$ here. 

Now, as with the special case $a=2$, the function $W_a(x)$ has large jump discontinuities at powers of $2$, of size $x^{\log a/\log 2}$. As a result, for infinitely many values of $x$, one has (cf.\ Equation \eqref{x2eq}) 
\begin{equation}\label{ax2eq}
    |W_a(x)|\geq\frac{x^{\log_2a}}{2},
\end{equation}
where we have used the standard notation $$\log_2 x =\frac{\log x}{\log 2}.$$ This leads us to conjecture the following, which generalises Conjecture \ref{suncon}.
\begin{conjecture}\label{newcon2}
    For all real numbers $a\ge 2$ and sufficiently large $x\geq x_0(a)$,
    \begin{equation*}
        |W_a(x)| <x^{\log_2a}.
    \end{equation*}
\end{conjecture}
Conjecture \ref{newcon2} agrees well with computational data, which we discuss in Section \ref{expsect}. By generalising a method of Tenenbaum \cite[Exercise 58]{tenenbaum2024solutions}, we are also able to obtain a result analogous to Theorem \ref{mainthm} and Corollary \ref{limsupcor} for $W_a(x)$ and $a>2$.
\begin{theorem}\label{limsupthm2}
    For all real numbers $a>2$, the limit supremum 
    \begin{equation}\label{sadef}
        s_a=\limsup_{x\to\infty}\frac{|W_a(x)|}{x^{\log_2a}}
    \end{equation}
    exists, and satisfies
    \begin{equation}\label{sabounds}
        \frac{1}{2}\leq s_a<\frac{a-1}{a+1}\cdot\frac{3^{v_a}-1}{3^{v_a}-a}\cdot\zeta(v_a),
    \end{equation}
    where $v_a=\log_2 a$ and $\zeta(\cdot)$ denotes the Riemann zeta function. For example, we have $1/2\leq s_{2.5}< 2.94$ and $1/2\leq s_{10}< 1.24$.
\end{theorem}
The proof of Theorem \ref{limsupthm2} is much shorter than Theorem \ref{mainthm} and does not require any explicit bounds on $M(x)$ or $m(x)$. We also note that the upper bound in \eqref{sabounds} approaches $1$ as $a\to\infty$, and approaches $\infty$ as $a\to 2^+$, highlighting that the case $a=2$ requires a different approach. In addition, the bound \eqref{sabounds} is not strong enough to prove Conjecture \ref{newcon2} for any value of $a$. However, if we fix a value of $a>2$, it is possible to compute $s_a$ to any desired precision. We demonstrate this for $a=3$, and consequently verify Conjecture \ref{newcon2} in this case.

\begin{theorem}\label{3thm}
    We have $s_3=0.813\pm 0.158<1$.
\end{theorem}

Note that the computational difficulty of determining a value for $s_a$ depends heavily of the choice of $a>2$. For instance, estimating $s_3$ is much easier than estimating $s_a$ for $a$ very close to $2$.

The outline of the rest of the paper is as follows. First, in Section \ref{oddsect} we bound the odd terms in $W(x)$, avoiding the large jump discontinuities at even values of $x$. In Section \ref{proofsect}, we prove Theorem \ref{mainthm} by modifying an inclusion-exclusion style result due to Grosswald \cite{grosswald1956average}, and in Section \ref{limsup2sec} we prove Theorems \ref{limsupthm2} and \ref{3thm}. Finally, further discussion and experimentation is given in Section \ref{expsect}. There, we explore avenues for further improvements to Theorem \ref{mainthm} and use computational data to better understand the behaviour of $W_a(x)$. An appendix is also included which updates existing explicit bounds for the functions $M(x)$ and $m(x)$. These may be of independent interest.

\section{Bounding the odd terms}\label{oddsect}
Inspired by Grosswald's \cite{grosswald1956average} treatment of the function $G(x)$, defined in \eqref{gxeq}, we will first study the odd terms of $W(x)$. So, define
\begin{equation}\label{bigUdef}
    U(x) =\sum_{\substack{n\leq x\\n\ \text{odd}}}(-2)^{\Omega(n)}.
\end{equation}
Compared to $W(x)$, the function $U(x)$ avoids the large jump discontinuities of size $O(x)$ at even values of $x$. In addition, we find that $U(x)$ is more amenable to a complex-analytic treatment than $W(x)$. 

Let $H(s)$ be the Dirichlet series
\begin{equation*}
    H(s)=\sum_{\substack{n\geq 1\\n\ \text{odd}}}\frac{(-2)^{\Omega(n)}}{n^s}=\prod_{p>2}\left(1+\frac{2}{p^s}\right)^{-1}=\prod_{p>2}\left(1-\frac{2}{p^s}+\frac{4}{p^{2s}}-\cdots\right).
\end{equation*}
We note that 
\begin{equation}\label{mu2eq}
    \frac{1}{\zeta^2(s)}=\left(\sum_{n=1}^\infty\frac{\mu(n)}{n^s}\right)^2,
\end{equation}
and thus 
\begin{equation}\label{approxz2}
    H(s)\approx(1-2^{1-s}+2^{-2s})^{-1}\cdot\frac{1}{\zeta^2(s)}=\prod_{p>2}\left(1-\frac{2}{p^s}+\frac{1}{p^{2s}}\right),
\end{equation}
in the sense that each factor of $H(s)$ agrees with the first two terms of each factor on the right-hand side of \eqref{approxz2}. Using \eqref{mu2eq} and \eqref{approxz2}, our general approach will be to utilise bounds on the function $m(x)$, defined in \eqref{littlemeq}, to obtain bounds for
\begin{equation}\label{littleueq}
    u(x)=\sum_{\substack{n\leq x\\ n\ \text{odd}}}\frac{(-2)^{\Omega(n)}}{n},
\end{equation}
which is related to $U(x)$ by partial summation.

In particular, the fact that $\lim_{x\to\infty} m(x)=0$ allows us to obtain bounds of the form $u(x)=o(1)$ and $U(x)=o(x)$. % a $o(1)$ bound for $u(x)$, and a $o(x)$ bound for $U(x)$. 
After obtaining a suitably strong bound for $U(x)$, we then use an inclusion-exclusion type result to bound $W(x)$ in Section \ref{proofsect}.

Now, to make the approximation in \eqref{approxz2} more explicit, we begin with the following lemma.
\begin{lemma}\label{fconlem}
    Let 
    \begin{equation}\label{Fdef}
        F(s)=H(s)\cdot(1-2^{1-s}+2^{-2s})\zeta^2(s).
    \end{equation}
    Then $F(s)$ converges absolutely in the half-plane $\Re(s)>\log 2/\log 3$.
\end{lemma}
\begin{proof}
    Note that
    \begin{align}
        F(s)&=H(s)\cdot(1-2^{1-s}+2^{-2s})\zeta^2(s)\notag\\
        &=\prod_{p>2}\left(1-\frac{2}{p^s}+\frac{4}{p^{2s}}-\cdots\right)\left(1+\frac{2}{p^s}+\frac{3}{p^{2s}}+\cdots\right)\notag\\
        &=\prod_{p>2}\left(1+\frac{3}{p^{2s}}-\frac{2}{p^{3s}}+\cdots\right).\label{fproduct}
    \end{align}
    Here, the lack of a $p^{-s}$ term in \eqref{fproduct} means we are not restricted to just $\Re(s)>1$, but we are allowed to define $F(s)$ in a wider half-plane (though we are still restricted to $\Re(s)>1/2$). However, we still have to determine the region of convergence for each factor in \eqref{fproduct}. So, to be more precise, we write
    \begin{equation}\label{fprodsum}
        F(s)=\prod_{p>2}\sum_{k=0}^\infty\frac{a(k)}{p^{ks}},
    \end{equation}
    where
    \begin{equation}\label{akeq}
        a(k)=\sum_{j=0}^k(-1)^{k-j}2^{k-j}(j+1)=(-2)^k\left[\sum_{j=0}^k\frac{(-1)^j}{2^j}+\sum_{j=0}^k(-1)^j\frac{j}{2^j}\right].
    \end{equation}
    By summing the geometric series, we see that
    \begin{equation}\label{23eq}
        \sum_{j=0}^\infty\frac{(-1)^j}{2^j}=\frac{2}{3}.
    \end{equation}
    Next, set 
    \begin{equation*}
        S_k=\sum_{j=0}^k(-1)^j\frac{j}{2^j}=\sum_{j=1}^k(-1)^j\frac{j}{2^j}.
    \end{equation*}
    Then
    \begin{align*}
        \frac{S_k}{2}&=\sum_{j=0}^k(-1)^j\frac{j}{2^{j+1}}\\
        &=\sum_{j=1}^{k+1}(-1)^{j-1}\frac{(j-1)}{2^j}\\
        &=-S_k+(-1)^k\frac{k}{2^{k+1}}+\sum_{j=1}^{k+1}\frac{(-1)^j}{2^j}.
    \end{align*}
    As a result,
    \begin{align}\label{29eq}
        \lim_{k\to\infty}S_k=\lim_{k\to\infty}\frac{2}{3}\left[(-1)^k\frac{k}{2^{k+1}}+\sum_{j=1}^{k+1}\frac{(-1)^j}{2^j}\right]=-\frac{2}{9}.
    \end{align}
    Substituting \eqref{23eq} and \eqref{29eq} into \eqref{akeq} gives
    \begin{equation*}
        a(k)\sim\frac{1}{9}(-2)^{k+2}.
    \end{equation*}
    Consequently, for any $p>2$, the sum 
    \begin{equation*}
        \sum_{k=0}^{\infty}\frac{a(k)}{p^{ks}},
    \end{equation*}
    appearing in \eqref{fprodsum}, converges absolutely for 
    \begin{equation}\label{logpeq}
        \Re(s)>\frac{\log 2}{\log p}\geq\frac{\log 2}{\log 3}.
    \end{equation}
    So, by extension, $F(s)$ converges absolutely in the half plane $\Re(s)>\log 2/\log 3$ as required.
\end{proof}
\begin{remark}
    The fact that $H(s)$ contains no even terms, and thus no $p=2$ factor, was essential to extending the region of convergence for $F(s)$ beyond $\Re(s)>1$. In \cite[Section 5]{mossinghoff2021oscillations}, Mossinghoff and Trudgian claimed a similar result for the function
    \begin{equation*}
        J(s)=\sum_{n\geq 1}\frac{(-2)^{\Omega(n)}}{n^s}=\prod_{p}\left(1+\frac{2}{p^s}\right)^{-1},
    \end{equation*}
    which includes both the odd and even terms. However, their claim of the region of convergence was erroneous. This was due to their lack of treatment of the asymptotics of the sequence $a(k)$, as in the above proof.
\end{remark}

Now that we have established a region of absolute convergence for $F$, we obtain a simple bound for the positive term series
\begin{equation}\label{fstareq}
    F^*(s)=\prod_{p>2}\sum_{k=0}^\infty\frac{|a(k)|}{p^{ks}},
\end{equation}
where $a(k)$ is as defined in \eqref{akeq}.
\begin{lemma}\label{fstarsigmalem}
    With $F^*(s)$ as defined in \eqref{fstareq}, we have
    \begin{equation*}
        F^*(\sigma)\leq\prod_{p>2}\left(1+\frac{3}{p^{2\sigma}}\cdot\frac{1}{1-2 p^{-\sigma}}\right)
    \end{equation*}
    for any real $\sigma>\log 2/\log 3$.
\end{lemma}
\begin{proof}
    Firstly, it follows from \eqref{akeq} that
    \begin{equation}\label{akabs}
        |a(k)|\leq 2^k\left|\sum_{j=0}^k\frac{(-1)^j(j+1)}{2^j}\right|.
    \end{equation}
    We then note that the sum
    \begin{equation*}
        A_k=\sum_{j=0}^k\frac{(-1)^j(j+1)}{2^j}
    \end{equation*}
    is an alternating series in which the absolute value of the summand, $(j+1)/2^j$, decreases towards zero. Therefore, $A_k$ is bounded above by its even partial sums, and bounded below by its odd partial sums. In particular, for all $k\geq 2$,
    \begin{equation*}
        0=A_1\leq A_k\leq A_2=\frac{3}{4}.
    \end{equation*}
    Substituting this into \eqref{akabs} gives
    \begin{equation*}
        |a(k)|\leq 3\cdot 2^{k-2} 
    \end{equation*}
    for all $k\geq 2$. Thus, recalling that $a(0)=1$ and $a(1)=0$, we have
    \begin{align*}
        F^*(\sigma)&=\prod_{p>2}\left(1+\sum_{k=2}^\infty\frac{a(k)}{p^{k\sigma}}\right)\\
        &\leq\prod_{p>2}\left(1+\frac{3}{p^{2\sigma}}\sum_{k=0}^\infty\frac{2^k}{p^{k\sigma}}\right)\\
        &=\prod_{p>2}\left(1+\frac{3}{p^{2\sigma}}\cdot\frac{1}{1-2p^{-\sigma}}\right)
    \end{align*}
    for any real $\sigma>\log 2/\log 3$, as required.
\end{proof}
\begin{proposition}\label{fstarprop}
    We have
    \begin{align}
        F^*(1)&\leq 2.8917,\label{fstarb1}\\ 
        F^*(0.88)&\leq 5.4772,\label{fstarb2}\\ 
        F^*(0.7)&\leq 66.568.\label{fstarb3}
    \end{align}
\end{proposition}
\begin{proof}
    We begin by bounding $F^*(1)$. By Lemma \ref{fstarsigmalem}, 
    \begin{equation}\label{fstar1eq}
        F^*(1)\leq\prod_{p>2}\left(1+\frac{3}{p^{2}}\cdot\frac{1}{1-2p^{-1}}\right)=\prod_{p>2}\left(1+\frac{3}{p(p-2)}\right).
    \end{equation}
    Let $m\geq 3$ and let $p_m$ denote the $m^{\text{th}}$ prime number. To estimate \eqref{fstar1eq}, we use the inequality $\log(1+x)<x$ for all $x>0$ to obtain
    \begin{align}
        \prod_{p>2}\left(1+\frac{3}{p(p-2)}\right)&=\prod_{2<p\leq p_m}\left(1+\frac{3}{p(p-2)}\right)\exp\left(\sum_{p>p_m}\log\left(1+\frac{3}{p(p-2)}\right)\right)\notag\\
        &\leq\prod_{2<p\leq p_m}\left(1+\frac{3}{p(p-2)}\right)\exp\left(\sum_{p>p_m}\frac{3}{p(p-2)}\right)\notag\\
        &\leq\prod_{2<p\leq p_m}\left(1+\frac{3}{p(p-2)}\right)\exp\left(\frac{3p_m}{p_m-2}\sum_{n>p_m}\frac{1}{n^2}\right)\notag\\
        &=\prod_{2<p\leq p_m}\left(1+\frac{3}{p(p-2)}\right)\exp\left(\frac{3p_m}{p_m-2}\left(\zeta(2)-\sum_{n\leq p_m}\frac{1}{n^2}\right)\right)\label{3pp2bound}.
    \end{align}
    By taking larger values of $m$ in \eqref{3pp2bound}, one obtains a better bound for $F^*(1)$. In particular, setting $m=10^4$ gives $F^*(1)\leq 2.8917$ as desired. More generally, in order to obtain a bound for $F^*(1-\theta)$ for some $\theta<1-\log 2/\log 3$, we use that
    \begin{align*}
        F^*(1-\theta)&\leq\prod_{p>2}\left(1+\frac{3}{p^{2(1-\theta)}}\cdot\frac{1}{1-2 p^{-1+\theta}}\right)\\
        &\qquad\qquad\boldsymbol{\cdot} \exp\left(\frac{3}{1-2p_m^{-1+\theta}}\left(\zeta(2-2\theta)-\sum_{n\leq p_m}\frac{1}{n^{2(1-\theta)}}\right)\right)
    \end{align*}
    for any $m\geq 3$. Setting $\theta=0.12$ and $m=10^4$ gives us $F^*(0.88)\leq 5.4772$. Then, setting $\theta=0.3$ and $m=10^4$ yields $F^*(0.7)\leq 66.568$.
\end{proof}

Next we move on to bounding the sum
\begin{equation}\label{m2def}
    m_2(x)=\sum_{n\leq x}\frac{(\mu*\mu)(n)}{n},
\end{equation}
where $\mu*\mu$ is the Dirichlet convolution of $\mu$ with itself. In particular, $m_2(x)$ relates to the factor of $1/\zeta^2(s)$ appearing in \eqref{approxz2}, since
\begin{equation}\label{1zeta2eq}
    \frac{1}{\zeta^2(s)}=\sum_{n\leq x}\frac{(\mu*\mu)(n)}{n^s}.
\end{equation}

\begin{lemma}\label{m2lem}
    Let $x_0 >1$. Suppose we have for all $x\geq x_0$, the bounds 
    \begin{equation}\label{c1c2eq}
        |m(\sqrt{x})|=\left|\sum_{n\leq \sqrt{x}}\frac{\mu(n)}{n}\right|\leq\frac{C_1}{(\log x)^{\alpha}},\quad\text{and}\quad \sum_{n\leq \sqrt{x}}\frac{|\mu(n)|}{n}\leq C_2\log x,
    \end{equation}
    for some $\alpha>0$ and constants $C_1=C_1(\alpha,x_0)$, $C_2=C_2(x_0)$. Then for all $x\geq x_0$,
    \begin{equation*}
        |m_2(x)|\leq\frac{2^{\alpha+1} C_1C_2}{(\log x)^{\alpha-1}}+\frac{2^{2\alpha}C_1^2}{(\log x)^{2\alpha}}.
    \end{equation*}
\end{lemma}
\begin{proof}
    First, using the Dirichlet hyperbola method one has
    \begin{equation}\label{hyperbolamu}
        m_2(x)=2\sum_{a\leq \sqrt{x}}\frac{\mu(a)}{a}\sum_{b\leq \frac{x}{a}}\frac{\mu(b)}{b}-\left(\sum_{a\leq \sqrt{x}}\frac{\mu(a)}{a}\right)^2.
    \end{equation}
    To conclude, we use \eqref{c1c2eq} and \eqref{hyperbolamu} to obtain
    \begin{align*}
        \left|m_2(x)\right| &\le 2\sum_{a\leq \sqrt{x}}\frac{|\mu(a)|}{a}\left|\sum_{b\le x/a}\frac{\mu(b)}{b}\right|+ \left(\sum_{a\leq \sqrt{x}}\frac{\mu(a)}{a}\right)^2 \\
        &\le 2\sum_{a\leq \sqrt{x}}\frac{|\mu(a)|}{a}\frac{C_1}{(\log\frac{x}{a})^\alpha}+ \frac{2^{2\alpha}C_1^2}{(\log x)^{2\alpha}} \\
        &\le \frac{2^{\alpha+1} C_1}{\log^\alpha x}\sum_{a\leq \sqrt{x}}\frac{|\mu(a)|}{a}+\frac{2^{2\alpha}C_1^2}{(\log x)^{2\alpha}} \\
        &\le\frac{2^{\alpha+1} C_1C_2}{(\log x)^{\alpha-1}}+\frac{2^{2\alpha}C_1^2}{(\log x)^{2\alpha}}
    \end{align*}
    as desired.
\end{proof}

\begin{proposition}\label{m2prop}
    We have the following bounds for $m_2(x)$:
    \begin{align}
        |m_2(x)|&\leq 2.06,\hspace{1.67cm} x\geq 1,\label{m2b1}\\
        |m_2(x)|&\leq\frac{57.88}{\log x},\qquad\qquad  x>1,\label{m2b2}\\
        |m_2(x)|&\leq\frac{117.67}{(\log x)^{1.35}},\qquad x\geq e^{780}.\label{m2b3}
    \end{align}
\end{proposition}
\begin{proof}
    First we note that our bounds on $m(x)$ in the Appendix (Theorem \ref{thm:lilmx}) can be rewritten as
    \begin{align}
        |m(\sqrt{x})|&\leq\frac{18.364}{(\log x)^2},\hspace{1.1cm} x>1,\label{msq1}\\
        |m(\sqrt{x})|&\leq\frac{37.712}{(\log x)^{2.35}},\qquad\, x\geq e^{780}.\label{msq2}
    \end{align}
    Then, by \cite[Corollary 3.3  and Lemma 9.2]{ramare2019explicit},
    \begin{align}
        \sum_{n\leq\sqrt{x}}\frac{|\mu(n)|}{n}&\leq 0.38\log x, \hspace{0.85cm} x\geq 10^6,\label{qb1}\\
        \sum_{n\leq\sqrt{x}}\frac{|\mu(n)|}{n}&\leq 0.306\log x,\qquad x\geq e^{780}.\label{qb2}
    \end{align}
    In the notation of Lemma \ref{m2lem} this means that we can set $C_1(2,10^6)=18.364$, $C_2(10^6)=0.38$, $C_1(2.35,e^{780})=37.712$, and $C_2(e^{780})=0.306$. So, applying Lemma \ref{m2lem} with these values, combined with a direct computation of $|m_2(x)|$ for $1\leq x<10^6$, gives \eqref{m2b1}--\eqref{m2b3} as desired.
\end{proof}
\begin{remark}
    Our decision to use $x=10^6$ as the cut-off point for our direct computation of $|m_2(x)|$ in Lemma \ref{m2prop} was somewhat arbitrary. Certainly, one could perform a larger computation to get a small improvement to \eqref{m2b1} and \eqref{m2b2}. However, the authors felt that $x=10^6$ was a good point to get a near optimal result whilst also making the calculation reasonable enough for any modern computer.  
\end{remark}

Equipped with bounds for $F^*(\sigma)$ and $m_2(x)$, we are now able to bound $u(x)$, defined in \eqref{littleueq}. For convenience, we also define the two functions
\begin{equation*}
    G(s)=\frac{F(s)}{1-2^{1-s}+2^{-2s}}
\end{equation*}
and
\begin{equation}\label{gstardef}
    G^*(s)=\frac{F^*(s)}{1-2^{1-s}+2^{-2s}}=\Bigg(\sum_{k\geq 0}\frac{1}{2^{ks}}\Bigg)^2F^*(s),
\end{equation}
with $F(s)$ and $F^*(s)$ as defined in \eqref{Fdef} and \eqref{fstareq}, respectively. By \eqref{Fdef}, this means that
\begin{equation}\label{gz2coneq}
    H(s)=\frac{G(s)}{\zeta^2(s)},
\end{equation}
so that the coefficients of $H(s)$ can be directly related to those of $G(s)$ and $1/\zeta^2(s)$ by Dirichlet convolution. Note that, similar to $F(s)$, the function $G(s)$ converges absolutely for $\Re(s)>\log 2/\log 3$ (see Lemma \ref{fconlem}).
\begin{lemma}\label{lilugenb}
    Let $u(x)$, $G^*(s)$, and $m_2(x)$ be as defined in \eqref{littleueq}, \eqref{gstardef}, and \eqref{m2def}, respectively. Let $K(x,x_0)$ be a decreasing function such that for all $x\geq x_0>1$,
    \begin{equation*}
        |m_2(x)|\leq K(x,x_0).
    \end{equation*}
    Then, for any $\theta\in\big[0,1-\frac{\log 2}{\log 3}\big]$ and $\eta\in[0,1]$, we have
    \begin{equation*}
        |u(x)|\leq K(x^{1-\eta},x_0)G^*(1)+\frac{2.06G^*(1-\theta)}{x^{\eta\theta}}
    \end{equation*}
    for all $x\geq x_0^{1/(1-\eta)}$.
\end{lemma}
\begin{proof}
    From \eqref{gz2coneq},
    \begin{equation}\label{hfeq}
        H(s):=\sum_{\substack{n\geq 1\\n\ \text{odd}}}\frac{(-2)^{\Omega(n)}}{n^s}=\frac{G(s)}{\zeta^2(s)},
    \end{equation}
    where $G(s)$ can be written as a Dirichlet series
    \begin{equation*}
        G(s)=\sum_{n\geq 1}\frac{g(n)}{n^s}
    \end{equation*}
    that converges for $\Re(s)>\log 2/\log 3$, and $G^*(s)$ (defined in \eqref{gstardef}) can be expressed as
    \begin{equation*}
        G^*(s)=\sum_{n\geq 1}\frac{|g(n)|}{n^s}.
    \end{equation*}
    Now, from \eqref{hfeq} and \eqref{1zeta2eq},
    \begin{align*}
        u(x)&=\sum_{n\leq x}\frac{(g*\mu*\mu)(n)}{n}\\
        &=\sum_{a\leq x}\frac{g(a)}{a}\sum_{b\leq x/a}\frac{(\mu*\mu)(b)}{b}.
    \end{align*}
    Hence, for $\eta\in[0,1]$ and $x\geq x_0^{1/(1-\eta)}$,
    \begin{align*}
        |u(x)|&\leq \sum_{a\leq x}\frac{|g(a)|}{a}\left|\sum_{b\leq x/a}\frac{(\mu*\mu)(b)}{b}\right|\\
        &\leq \sum_{a\leq x^{\eta}}\frac{|g(a)|}{a}\left|\sum_{b\leq x/a}\frac{(\mu*\mu)(b)}{b}\right|+\sum_{x^{\eta}<a\leq x}\frac{|g(a)|}{a}\left|\sum_{b\leq x/a}\frac{(\mu*\mu)(b)}{b}\right|\\
        &\leq G^*(1)K(x^{1-\eta},x_0)+2.06\sum_{x^{\eta}<a\leq x}\frac{|g(a)|}{a},
    \end{align*}
    where in the last line we used \eqref{m2b1} to bound the second sum over $(\mu*\mu)(b)/b$. To conclude, we note that for $a>x^{\eta}$
    \begin{equation*}
        \frac{1}{a}=\frac{1}{a^{\theta}a^{1-\theta}}\leq\frac{1}{x^{\eta\theta}a^{1-\theta}},
    \end{equation*}
    hence
    \begin{equation*}
        \sum_{x^{\eta}<a\leq x}\frac{|g(a)|}{a}\leq\frac{1}{x^{\eta\theta}}\sum_{x^{\eta}<a\leq x}\frac{|g(a)|}{a^{1-\theta}}\leq\frac{1}{x^{\eta\theta}}\sum_{a\geq 1}\frac{|g(a)|}{a^{1-\theta}}=\frac{G^*(1-\theta)}{x^{\eta\theta}}.\qedhere
    \end{equation*}
\end{proof}

\begin{proposition}\label{liluprop}
    We have the following bounds on $u(x)$:
    \begin{align}
        |u(x)|&\leq\frac{1501.93}{\log x},\,\hspace{1cm} x>1,\label{lilub1}\\
        |u(x)|&\leq\frac{812.59}{\log x},\,\hspace{1.2cm} x\geq e^{195},\label{lilub2}\\
        |u(x)|&\leq\frac{1714.26}{(\log x)^{1.35}},\qquad x\geq e^{995}\label{lilub3}.
    \end{align}
\end{proposition}
\begin{proof}
    First we prove $\eqref{lilub1}$. In the notation of Lemma \ref{lilugenb}, let $\eta=0.33$, $\theta=0.12$ and $x_0=2$. We also set $K(x,x_0)=57.88/\log x$ by Proposition \ref{m2prop}. Then, for all $x\geq 3$, Lemma \ref{lilugenb} and Proposition \ref{fstarprop} give
    \begin{align*}
        |u(x)|&\leq \frac{57.88}{0.67\log x}\cdot\frac{2.8917}{0.25}+\frac{2.06}{x^{0.0396}}\cdot\frac{5.8925}{1-2^{0.88}+2^{-1.76}}\\
        &\leq\frac{1501.93}{\log x},
    \end{align*}
    upon noting that 
    \begin{equation*}
        \left(\frac{57.88}{0.67\log x}\cdot\frac{2.8917}{0.25}+\frac{2.06}{x^{0.0396}}\cdot\frac{5.8925}{1-2^{0.88}+2^{-1.76}}\right)\log x
    \end{equation*}
    is maximised at $x=\exp(1/\eta\theta)\approx 10^{11}$. This proves \eqref{lilub1} for $x\geq 3$. The case $1<x\leq 3$ can then
    be checked by hand.

    The bounds \eqref{lilub2} and \eqref{lilub3} are obtained in an analogous way. To prove \eqref{lilub2} our choice of parameters for Lemma \ref{lilugenb} are $\eta=0.16$, $\theta=0.3$, $x_0=e^{195}$ and again, $K(x,x_0)=57.88/\log x$. Finally, for \eqref{lilub3} we use $\eta=1-\frac{790}{995}$, $\theta=0.3$, $x_0=e^{790}$ and $K(x,x_0)=117.67/(\log x)^{1.35}$. 
\end{proof}
\begin{remark}
    For each bound \eqref{lilub1}--\eqref{lilub3} in the above proof, our choices of $\eta$ and $\theta$ are optimised to 2 decimal places in order to obtain the smallest numerators in \eqref{lilub1}--\eqref{lilub3}. The only exception to this is our choice of $\eta$ for the final $|u(x)|$ bound \eqref{lilub3}, whereby $\eta=1-\frac{790}{995}$ is chosen as it is the largest possible value of $\eta$ that can be used to obtain a result that holds for all $x\geq e^{995}$.
\end{remark}

Finally, via partial summation, we are able to obtain bounds for $U(x)$. We first provide two short lemmas, which will be used to deal with the integral that occurs in the process of partial summation.

\begin{lemma}[{\cite[Lemma 5.9]{bennett2018explicit}}]\label{benlem}
    For all $x\geq 1865$,
    \begin{equation*}
        \int_2^x\frac{1}{\log t}\mathrm{d}t<\frac{x}{\log x}+\frac{3x}{2\log^2 x}.
    \end{equation*}
\end{lemma}
\begin{lemma}[{cf.\ \cite[Lemma 7]{bordelles2015some}}]\label{bordlem}
    Let $a>1$ and $\alpha>0$ be real numbers. For all $x\geq a$,
    \begin{equation*}
        \int_a^x\frac{1}{(\log t)^{\alpha}}\mathrm{d}t\leq\frac{C_{\alpha}x}{(\log x)^{\alpha}},
    \end{equation*}
    where 
    \begin{equation*}
        C_{\alpha}=\alpha^{-1}\left(\frac{\alpha}{(e\log a)^{\alpha/(\alpha+1)}}+\alpha^{1/(\alpha+1)}\right)^{\alpha+1}.
    \end{equation*}
\end{lemma}
\begin{proof}
    The proof is a straightforward modification of the proof of \cite[Lemma 7]{bordelles2015some}, but we include the details here for completeness. Firstly, for any $b>1$,
    \begin{equation}\label{intaxsplit}
        \int_a^x\frac{1}{(\log t)^{\alpha}}\mathrm{d}t=\left(\int_a^{x^{1/b}}+\int_{x^{1/b}}^x\right)\frac{1}{(\log t)^{\alpha}}\mathrm{d}t\leq\frac{x^{1/b}}{(\log a)^{\alpha}}+\frac{b^{\alpha}x}{(\log x)^{\alpha}}.
    \end{equation}
    The inequality $\log x\leq Ce^{-1}x^{1/C}$, used with $C=b\alpha/(b-1)$, then gives
    \begin{equation*}
        (\log x)^{\alpha}\leq\left(\frac{b\alpha}{e(b-1)}\right)^{\alpha}x^{1-1/b},
    \end{equation*}
    and thus
    \begin{equation}\label{x1beq}
        x^{1/b}\leq\left(\frac{b\alpha}{e(b-1)}\right)^{\alpha}\frac{x}{(\log x)^{\alpha}}.
    \end{equation}
    Substituting \eqref{x1beq} into \eqref{intaxsplit} yields
    \begin{equation*}
        \int_a^x\frac{1}{(\log t)^{\alpha}}\mathrm{d}t\leq x\left(\frac{b}{\log x}\right)^{\alpha}\left(1+\left(\frac{\alpha}{e(b-1)\log a}\right)^{\alpha}\right)    ,
    \end{equation*}
    so that choosing
    \begin{equation*}
        b=1+\left(\frac{\alpha}{e\log a}\right)^{\alpha/(\alpha+1)}
    \end{equation*}
    gives the desired result.
\end{proof}

\begin{proposition}\label{biguprop}
    Let $U(x)$ be as defined in \eqref{bigUdef}. We have:
    \begin{align}
        |U(x)|&\leq\frac{3071.82x}{\log x},\hspace{0.85cm} x\geq 2.5\cdot 10^{14},\label{bigUb1}\\
        |U(x)|&\leq\frac{1636.07x}{\log x},\hspace{0.85cm} x\geq e^{200},\label{bigUb2}\\
        |U(x)|&\leq\frac{3541.86x}{(\log x)^{1.35}},\qquad x\geq e^{1000}.\label{bigUb3}
    \end{align}
\end{proposition}
\begin{proof}
    With $u(x)$ as defined in \eqref{littleueq}, we have by partial summation
    \begin{equation*}
        U(x)=x\cdot u(x)-\int_{1}^xu(t)\,\mathrm{d}t
    \end{equation*}
    for all $x>1$. Therefore
    \begin{equation}\label{partsumU}
        |U(x)|\leq x|u(x)|+\int_1^x\left|u(t)\right|\mathrm{d}t.
    \end{equation}
    
    We begin by proving \eqref{bigUb1}. Let $x\geq 2.5\cdot 10^{14}$. Using Proposition \ref{liluprop}, we substitute the bound \eqref{lilub1} 
    into \eqref{partsumU} to obtain
    \begin{align}\label{Uinteq}
        |U(x)|\leq\frac{1501.93x}{\log x}+\int_{1}^2|u(t)|\mathrm{d}t+\int_2^x\frac{1501.93}{\log t}\mathrm{d}t.
    \end{align}
    Since $u(1)=1$, the first integral in \eqref{Uinteq} is equal to $1$. Then by Lemma \ref{benlem}, the second integral can be bounded:
    \begin{equation*}
        1501.93\int_{2}^x\frac{1}{\log t}\mathrm{d}t\leq 1501.93\left(\frac{x}{\log x}+\frac{3x}{2(\log x)^2}\right).
    \end{equation*}
    Substituting these bounds back into \eqref{Uinteq} gives
    \begin{equation*}
        |U(x)|\leq\frac{1501.93x}{\log x}+1+\frac{1501.93x}{\log x}+\frac{2252.895x}{(\log x)^2}\leq\frac{3071.82x}{\log x}
    \end{equation*}
    as desired. The proof of \eqref{bigUb2} is similar, except that we split the integral in \eqref{partsumU} even further. Namely, using \eqref{lilub1} and \eqref{lilub2} from Proposition \ref{liluprop}, we have for $x\geq e^{200}$,
    \begin{align*}
        |U(x)|&\leq\frac{812.59x}{\log x}+\int_1^2|u(t)|\mathrm{d}t+\int_2^{ e^{195}}|u(t)|\mathrm{d}t+\int_{ e^{195}}^{x}|u(t)|\mathrm{d}t\\
        &=\frac{812.59x}{\log x}+1+1501.93\int_2^{ e^{195}}\frac{1}{\log t}\mathrm{d}t+812.59\int_{ e^{195}}^{x}\frac{1}{\log t}\mathrm{d}t\\
        &=\frac{812.59x}{\log x}+1+(1501.93-812.59)\int_2^{ e^{195}}\frac{1}{\log t}\mathrm{d}t+812.59\int_{2}^{x}\frac{1}{\log t}\mathrm{d}t\\
        &\leq\frac{812.59x}{\log x}+1.731\cdot 10^{85}+812.59\int_2^x\frac{1}{\log t}\mathrm{d}t,
    \end{align*}
    so that, again using Lemma \ref{benlem},
    \begin{equation*}
        |U(x)|\leq \frac{812.59x}{\log x}+1.731\cdot 10^{85}+812.59\left(\frac{x}{\log x}+\frac{3x}{2(\log x)^2}\right)\leq\frac{1636.07x}{\log x}
    \end{equation*}
    as desired. Finally, we prove \eqref{bigUb3}. Let $x\geq\exp(1000)$. As before, we similarly substitute \eqref{lilub1}, \eqref{lilub2}, and \eqref{lilub3} into \eqref{partsumU} to obtain
    \begin{align}
        |U(x)|&\leq \frac{1714.26x}{(\log x)^{1.35}}+1+\int_{2}^{e^{195}}\frac{1501.93}{\log t}\mathrm{d}t+\int_{e^{195}}^{e^{995}}\frac{812.59}{\log t}\mathrm{d}t+\int_{e^{995}}^{x}\frac{1714.26}{(\log t)^{1.35}}\mathrm{d}t\notag\\
        &\leq \frac{1714.26x}{(\log x)^{1.35}}+1.0852\cdot 10^{432}+\frac{1765.793x}{(\log x)^{1.35}}\label{bordline}\\
        &\leq\frac{3541.86x}{(\log x)^{1.35}},\notag
    \end{align}
    where in \eqref{bordline} we used Lemma \ref{bordlem} to bound the integral involving $1/(\log t)^{1.35}$.
\end{proof}

\section{Proof of Theorem \ref{mainthm}}\label{proofsect}
In order to prove Theorem \ref{mainthm}, we need to link our estimates for $U(x)$ to an estimate for $W(x)$. This will be done using the following result, which we prove by successively splitting sums into their odd and even components.

\begin{lemma}\label{wklem}
    For any $x\geq 0$ and $k\geq 1$, we have 
    \begin{equation}\label{WUrelation}
        W(x) = U(x)-2U\left(\frac{x}{2}\right)+4U\left(\frac{x}{4}\right)-\cdots+(-2)^{k-1}U\left(\frac{x}{2^{k-1}}\right)+ (-2)^k W\left(\frac{x}{2^k}\right).
    \end{equation}
\end{lemma}
\begin{proof}
    To begin, we note that
    \begin{align}\label{woddeven}
        W(x) =\sum_{\substack{n\le x\\ n \text{ odd}}} (-2)^{\Omega(n)} + \sum_{\substack{n\le x\\ n \text{ even}}} (-2)^{\Omega(n)} = U(x)+\sum_{\substack{n\le x\\ n \text{ even}}} (-2)^{\Omega(n)}.
    \end{align}
    Now, write each even $n\geq 2$ as $n=2j$ and sum over $j$. In particular, since $\Omega(2j)=1+\Omega(j)$,
    \begin{align}
        \sum_{\substack{n\le x \\ n \text{ even}}}(-2)^{\Omega(n)}&= -2\sum_{\substack{j\le \frac{x}{2}}}(-2)^{\Omega(j)}\notag\\
        &=-2\left[U\left(\frac{x}{2}\right)+ \sum_{\substack{j\le \frac{x}{2}\\ j \text{ even}}} (-2)^{\Omega(j)}\right]\notag\\
        &=-2\left[U\left(\frac{x}{2}\right) -2 \left(U\left(\frac{x}{4}\right)+  \sum_{\substack{j\le \frac{x}{4}\\ j \text{ even}}} (-2)^{\Omega(j)}\right)\right]\notag\\
        &=-2U\left(\frac{x}{2}\right)+4U\left(\frac{x}{4}\right)-\cdots +(-2)^k W\left(\frac{x}{2^k}\right).\label{oddevenit}
    \end{align}
    Substituting \eqref{oddevenit} into \eqref{woddeven}, we obtain the desired result.
\end{proof}
\begin{remark}
    The above lemma was inspired by the work of Grosswald \cite{grosswald1956average}, who used an analagous identity for the summatory function of $2^{\Omega(n)}$.
\end{remark}

To deal with the final term $W(x/2^k)$ in \eqref{WUrelation}, we will use the following computational results of Mossinghoff and Trudgian \cite{mossinghoff2021oscillations}.
\begin{lemma}[{\cite[Section 5]{mossinghoff2021oscillations}}]\label{Mosscomp}
    We have
    \begin{align}
        |W(x)|&<x, &3078\leq x\leq 2.5\cdot 10^{14},\label{sunver}\\
        |W(x)|&<0.979x, &1.25\cdot 10^{14}\leq x\leq 2.5\cdot 10^{14}.\label{sunsharpver}
    \end{align}
\end{lemma}
We now give a bound for $|W(x)|$, dependent on the style of bounds for $|U(x)|$ obtained in Section \ref{oddsect}. Our result is stated in a general form to show clearly the impact bounds for $|U(x)|$ of different orders 
%different kinds of bounds on $|U(x)|$ 
have on a bound for $|W(x)|$. Notably, a result of the form 
\begin{equation*} 
    U(x)=O\left(\frac{x}{(\log x)^{^{1+\varepsilon}}}\right),    
\end{equation*}
for some $\varepsilon>0$, is necessary to obtain a bound of the form $W(x)=O(x)$.
\begin{proposition}\label{Wprop}
    Let $\varepsilon>0$. Suppose that, for constants $C_1,C_2,C_3>0$,
    \begin{align}
        |U(y)|&\leq \frac{C_1y}{\log y}, \,\qquad\qquad y\geq 2.5\cdot 10^{14},\label{14Ubound}\\
        |U(y)|&\leq \frac{C_2y}{\log y}, \,\qquad\qquad y\geq  e^{200},\label{200Ubound}\\
        |U(y)|&\leq \frac{C_3y}{(\log y)^{1+\varepsilon}}, \qquad y\geq e^{1000}.\label{1000Ubound}
    \end{align}
    Then, for all $x\geq 2.5\cdot 10^{14}$,
    \begin{equation*}
        |W(x)|<0.979x+E(x),
    \end{equation*}
    where 
    \begin{equation}\label{mainWbound}
        E(x)=
        \begin{cases}
            0.27C_1x,& x\in[2.5\cdot 10^{14},e^{200}),\\
            0.27C_1x+0.0782C_2x,& x\in[e^{200},e^{1000}),\\
            0.27C_1x+0.0782C_2x+\big(\frac{C_3x}{\log 2}\big)\frac{1}{\varepsilon(1442\log 2)^{\varepsilon}},& x\geq  e^{1000}.
        \end{cases}
    \end{equation}
\end{proposition}
\begin{proof}
     Let 
     \begin{equation}\label{kdef}
         k=\left\lceil\frac{\log x}{\log 2}-\frac{\log(2.5\cdot 10^{14})}{\log 2}\right\rceil
     \end{equation} 
     so that
     \begin{equation}\label{kbounds}
         \frac{\log x}{\log 2}-\frac{\log(2.5\cdot 10^{14})}{\log 2}\leq k<\frac{\log x}{\log 2}-\frac{\log(2.5\cdot 10^{14})}{\log 2}+1.
     \end{equation} It then follows that
    \begin{align}
        &1.25\cdot 10^{14}<\frac{x}{2^k}\leq 2.5\cdot 10^{14}\label{x2kbound},
    \end{align}
    and thus by \eqref{sunsharpver} in Lemma \ref{Mosscomp},
    \begin{equation}\label{W2kbound}
        \quad 2^k\left|W\left(\frac{x}{2^k}\right)\right|<0.979x.
    \end{equation}
    Substituting both \eqref{W2kbound} and our choice of $k$ from \eqref{kdef} into Lemma \ref{wklem} gives
    \begin{equation}\label{firstWexp}
        \left|W(x)\right|< \left|U(x)\right|+2\left|U\left(\frac{x}{2}\right)\right|+4\left|U\left(\frac{x}{4}\right)\right|+\cdots+2^{k-1}\left|U\left(\frac{x}{2^{k-1}}\right)\right|+0.979x.
    \end{equation}
    We now use our bounds on $U(x)$. First, let 
    \begin{align*}
        \ell_1=
        \begin{cases}
            k,&\text{if}\ 2.5\cdot 10^{14}\leq x<e^{200},\\
            k-\left\lceil\frac{\log x}{\log 2}-\frac{200}{\log(2)}\right\rceil,&\text{if}\ x\geq e^{200}.
        \end{cases}
    \end{align*}
    By the definition of $k$ in \eqref{kdef}, we find that
    \begin{align}
        0&\leq \ell_1\leq 241,\qquad\text{if}\ 2.5\cdot 10^{14}\leq x<e^{200},\notag\\
        240&\leq \ell_1\leq 241,\qquad\text{if}\ x\geq e^{200}.\label{secondmbound}
    \end{align}
    We then use \eqref{14Ubound} to obtain
    \begin{equation}\label{firstc1bound}
        \sum_{j=k-\ell_1}^{k-1}2^{j}\left|U\left(\frac{x}{2^j}\right)\right|\leq\sum_{j=k-\ell_1}^{k-1}C_1\frac{x}{\log(x/2^j)}=C_1x\sum_{j=k-\ell_1}^{k-1}\frac{1}{\log(x/2^j)}.
    \end{equation}
    Now, for any $r\in\{1,2,\ldots,\ell_1\}$ we have by \eqref{x2kbound} that
    \begin{equation*}
        \frac{1}{\log(x/2^{k-r})}=\frac{1}{r\log(2)\log(x/2^{k})}<\frac{1}{r\log(2)\log(1.25\cdot 10^{14})}.
    \end{equation*}
    As a result, we can bound the right-hand side of \eqref{firstc1bound} by
    \begin{equation}\label{secondc1bound}
        C_1x\sum_{j=k-\ell_1}^{k-1}\frac{1}{\log(x/2^j)}<\frac{C_1x}{\log(2)\log(1.25\cdot 10^{14})}\sum_{r=1}^{241}\frac{1}{r}< 0.27C_1x.
    \end{equation}
    Thus,
    \begin{equation}\label{c1bound}
        \sum_{j=k-\ell_1}^{k-1}2^{j}\left|U\left(\frac{x}{2^j}\right)\right|< 0.27C_1x.
    \end{equation}
    
    If $2.5\cdot 10^{14}\leq x<e^{200}$, then $\ell_1=k$ and substituting \eqref{c1bound} into \eqref{firstWexp} gives the first bound in \eqref{mainWbound}. Now suppose that $x\geq e^{200}$. Let
    \begin{align*}
        \ell_2=
        \begin{cases}
            k,&\text{if}\  e^{200}\leq x<e^{1000},\\
            k-\left\lceil\frac{\log x}{\log 2}-\frac{1000}{\log(2)}\right\rceil,&\text{if}\ x\geq e^{1000},
        \end{cases}
    \end{align*}
    so that 
    \begin{align}
        241&\leq \ell_2\leq 1395,\qquad\text{if}\  e^{200}\leq x<e^{1000},\notag\\
        1394&\leq \ell_2\leq 1395,\qquad\text{if}\ x\geq e^{1000}.\label{secondm2bound}
    \end{align}
    In the case where $\ell_2=\ell_1$ (e.g., if $x=e^{200}$), then 
    \begin{equation}\label{c2easy}
        \sum_{j=k-\ell_2}^{k-\ell_1}2^j\left|U\left(\frac{x}{2^j}\right)\right|=0<0.0782C_2x.
    \end{equation}
    Otherwise, if $\ell_2>\ell_1$, we can argue similarly as before to obtain
    \begin{align}
        \sum_{j=k-\ell_2}^{k-\ell_1-1}2^{j}\left|U\left(\frac{x}{2^j}\right)\right|&\leq C_2x\sum_{j=k-\ell_2}^{k-\ell_1-1}\frac{1}{\log(x/2^j)}\notag\\
        &<\frac{C_2x}{\log(2)\log(1.25\cdot 10^{14})}\sum_{r=\ell_1+1}^{\ell_2}\frac{1}{r}\notag\\
        &\leq\frac{C_2x}{\log(2)\log(1.25\cdot 10^{14})}\sum_{r=241}^{1395}\frac{1}{r}\notag\\
        &<0.0782C_2x.\label{c2bound}
    \end{align} 
    Combining \eqref{c1bound}, \eqref{c2easy} and \eqref{c2bound} then proves the second bound in \eqref{mainWbound}. 
    
    Finally, we consider the case $x\geq e^{1000}$. If $k=\ell_2$, then \eqref{c1bound} and \eqref{c2bound} account for all the terms in \eqref{firstWexp}. Otherwise, if $k>\ell_2$ we can use \eqref{1000Ubound} 
    to account for the remaining terms. In particular, by \eqref{1000Ubound} and \eqref{secondm2bound},
    \begin{equation}\label{firstc3bound}
        \sum_{j=0}^{k-\ell_2-1}2^{j}\left|U\left(\frac{x}{2^j}\right)\right|\leq C_3x\sum_{j=0}^{k-1395}\frac{1}{\left(\log\left(\frac{x}{2^j}\right)\right)^{1+\varepsilon}}.
    \end{equation}
    To approximate this sum, we note that
    \begin{equation*}
        0\leq\sum_{j=0}^{k-1395}\frac{1}{\left(\log\left(\frac{x}{2^j}\right)\right)^{1+\varepsilon}}\leq\int_0^{k-1395}\frac{1}{\left(\log\left(\frac{x}{2^t}\right)\right)^{1+\varepsilon}}\mathrm{d}t. 
    \end{equation*}
    Now,
    \begin{align}
         \int_0^{k-1395}\frac{1}{\left(\log\left(\frac{x}{2^t}\right)\right)^{1+\varepsilon}}\mathrm{d}t&=\frac{1}{\log 2}\left[\frac{1}{\varepsilon(\log x-t\log 2)^{\varepsilon}}\right]_{t=0}^{t=k-1395}\notag\\
         &<\left(\frac{1}{\log 2}\right)\frac{1}{\varepsilon(\log x-(k-1395)\log 2)^{\varepsilon}}\notag\\
         &\leq\left(\frac{1}{\log 2}\right)\frac{1}{\varepsilon(1442\log 2)^{\varepsilon}}\label{intbound},
    \end{align}
    where in the last line we used
    \begin{equation*}
        k-1395<\frac{\log x}{\log 2}-\frac{\log(2.5\cdot 10^{14})}{\log 2}-1394<\frac{\log x}{\log 2}-1442
    \end{equation*}
    by \eqref{kbounds}. Substituting \eqref{intbound} into \eqref{firstc3bound} gives
    \begin{equation}\label{c3bound}
        \sum_{j=0}^{k-\ell_2-1}2^{j}\left|U\left(\frac{x}{2^j}\right)\right|<\left(\frac{C_3x}{\log 2}\right)\frac{1}{\varepsilon(1442\log 2)^{\varepsilon}}.
    \end{equation}
    Combining \eqref{c3bound} with our results for the case $x<e^{1000}$ then proves the final inequality in \eqref{mainWbound}.
\end{proof}

Substituting our bounds for $U(x)$ from Section \ref{oddsect} into Proposition \ref{Wprop} now allows us to complete the proof of Theorem \ref{mainthm}.

\begin{proof}[Proof of Theorem \ref{mainthm}]
    To begin with, by a direct computation we have
    \begin{equation*}
        |W(x)|<2x,\qquad 1\leq x< 3078.
    \end{equation*}
    Then, Lemma \ref{Mosscomp} gives
    \begin{equation*}
        |W(x)|<x,\qquad 3078\leq x\leq 2.5\cdot 10^{14}.
    \end{equation*}
    Now, by Proposition \ref{biguprop} we can set $\varepsilon=0.35$, $C_1=3071.82$, $C_2=1636.07$, and $C_3=3541.86$ in Proposition \ref{Wprop}. This gives
    \begin{equation*}
        |W(x)|<
        \begin{cases}
            831x,& 2.5\cdot 10^{14}<x<e^{200},\\
            959x,& e^{200}\leq x< e^{1000},\\
            2260x,& x\geq  e^{1000}.
        \end{cases}
    \end{equation*}
    Combining our results for each range of $x$ completes the proof.
\end{proof}

\section{Proof of Theorems \ref{limsupthm2} and \ref{3thm}}\label{limsup2sec}
We now prove Theorem \ref{limsupthm2}, and then specialise to the case $a=3$ to prove Theorem \ref{3thm}. Our approach generalises Exercise $58$ in\footnote{The authors thank G. Tenenbaum for bringing this result to their attention.} \cite{tenenbaum2024solutions}, which instead considers the sum
\begin{equation*}
    \sum_{n\leq x}3^{\Omega(n)}.
\end{equation*}

From here onwards, we denote
\begin{equation*}
    P(a):= \prod_{p\leq a}p
\end{equation*}
as the product over all primes not exceeding $a$, and $p_k$ as the largest prime such that $p_k\leq a$. We will also make frequent use the of notation $v_a=\log_2 a=\log a/\log 2$ as in the statement of Theorem \ref{limsupthm2}.

\subsection{Proof of Theorem \ref{limsupthm2}}

To begin, we give the following result, which is analogous to Lemma \ref{wklem} in that it relates $W_a(x)$ to the sum
\begin{equation*}
    \sum_{\substack{m\leq x\\ (m,P(a))=1}}(-a)^{\Omega(m)}.
\end{equation*}
\begin{lemma}\label{wainexlem}
    For any real number $a>2$, we have
    \begin{equation}\label{wafulleq}
        W_a(x)=\sum_{\substack{m\leq x\\(m,P(a))=1}}(-a)^{\Omega(m)}\sum_{\substack{\alpha_i\geq 0\\ 2^{\alpha_1}3^{\alpha_2}5^{\alpha_3}\cdots p_k^{\alpha_k}\leq x/m}}(-a)^{\sum_{i=1}^k \alpha_i}.
    \end{equation}
\end{lemma}
\begin{proof}
    For any $n\leq x$, we write $n=m\prod_{i=1}^kp_i^{\alpha_{i}}$ so that $(m,P(a))=1$. The result then follows by the additivity of $\Omega(n)$.
\end{proof}
Next, we provide a bound for the inner sum in \eqref{wafulleq}. 
\begin{lemma}\label{innerupperlem}
    For any real $a>2$ and $z\geq 1$,
    \begin{equation}\label{ageoeq}
        \left|\sum_{\substack{\alpha_i\geq 0\\2^{\alpha_1}\cdots p_k^{\alpha_k}\leq z}}(-a)^{\sum_{i=1}^k\alpha_i}\right|\leq\frac{a}{a+1}z^{v_a}\prod_{i=2}^k\frac{1}{1-\frac{a}{p_i^{v_a}}}+O(z^{r_a}),
    \end{equation}
    where
    \begin{equation}\label{razeq}
        r_a=
        \begin{cases}
            0&\text{if}\ 2<a<3,\\
            \log_3a&\text{if}\ a\geq 3.
        \end{cases}
    \end{equation}
    for some constant $C_a>0$. Here we use the convention that if $k=1$, the product in \eqref{ageoeq} is equal to $1$.
\end{lemma}
\begin{proof}[Proof of Lemma \ref{innerupperlem}]
    If $2<a<3$, then $p_k=2$ and the result follows directly as the sum of a geometric series with an $O(1)$ error term. Otherwise, if $a\geq 3$, we write
    \begin{align}
        \sum_{\substack{\alpha_i\geq 0\\2^{\alpha_1}\cdots p_k^{\alpha_k}\leq z}}(-a)^{\sum_{i=1}^k\alpha_i}=\sum_{3^{\alpha_2}\cdots p_k^{\alpha_k}\leq z}(-a)^{\sum_{i=2}^k\alpha_i}\sum_{2^{\alpha_1}\leq z/(3^{\alpha_2}\cdots p_k^{\alpha_k})}(-a)^{\alpha_1}.\label{sumadecomp}
    \end{align}
    The right-hand side of \eqref{sumadecomp} is equivalent to
    \begin{align}
        \sum_{3^{\alpha_2}\cdots p_k^{\alpha_k}\leq z}(-a)^{\sum_{i=2}^k\alpha_i}\left(\frac{1}{a+1}-\frac{(-a)^{\lfloor\log_2(z/(3^{\alpha_2\cdots p_k^{\alpha_k}}))\rfloor+1}}{a+1}\right).\label{minasimpeq}
    \end{align}
    Now, 
    \begin{align*}
        &\left|\sum_{3^{\alpha_2}\cdots p_k^{\alpha_k}\leq z}(-a)^{\sum_{i=2}^k\alpha_i}\left(-\frac{(-a)^{\lfloor\log_2(z/(3^{\alpha_2\cdots p_k^{\alpha_k}}))\rfloor+1}}{a+1}\right)\right|\\
        &\qquad\qquad\leq\frac{a}{a+1}z^{v_a}\sum_{3^{\alpha_2}\cdots p_k^{\alpha_k}\leq z^{v_a}}\frac{a^{\sum_{i=2}^k \alpha_i}}{(3^{\alpha_2}\cdots p_k^{\alpha_k})^{v_a}}\\
        &\qquad\qquad\leq\frac{a}{a+1}z^{v_a}\sum_{\alpha_2=0}^\infty\frac{a^{\alpha_2}}{(3^{v_a})^{\alpha_2}}\sum_{\alpha_3=0}^\infty\frac{a^{\alpha_3}}{(5^{v_a})^{\alpha_3}}\cdots\sum_{\alpha_k=0}^\infty\frac{a^{\alpha_k}}{(p_k^{v_a})^{\alpha_k}}\\
        &\qquad\qquad=\frac{a}{a+1}z^{v_a}\prod_{i=2}^k\frac{1}{1-\frac{a}{p_i^{v_a}}},
    \end{align*}
    where in the last line we have used that $p^{v_a}=p^{\log_2 a}>a$ for all $p\geq 3$. This bounds the main term \eqref{minasimpeq}. The remainder in \eqref{minasimpeq} can then be iteratively bounded in the same way, yielding 
    \begin{align*}
        \left|\frac{1}{a+1}\sum_{3^{\alpha_2}\cdots p_k^{\alpha_k}\leq z}(-a)^{\sum_{i=2}^k\alpha_i}\right|&\leq\left(\frac{1}{(a+1)^k}+a\sum_{j=2}^k\frac{z^{\log_{p_j}a}}{(a+1)^j}\cdot \prod_{i=j+1}^k\frac{1}{1-\frac{a}{p_i^{v_a}}}\right)\\
        &=O(z^{\log_3 a}),
    \end{align*}
    which proves the lemma.
\end{proof}
We now give another lemma which will be needed to bound our final error term.
\begin{lemma}\label{outersumlem}
    For any real $a>2$,
    \begin{equation*}
        \sum_{\substack{m\leq x\\(m,P(a))=1}}\frac{a^{\Omega (m)}}{m}=O\big((\log x)^a\big).
    \end{equation*}
\end{lemma}
\begin{proof}
    Note that 
    \begin{equation*}
         \sum_{\substack{m\leq x\\(m,P(a))=1}}\frac{a^{\Omega (m)}}{m}\leq\prod_{p_k<p\leq x}\frac{1}{1-\frac{a}{p}}.
    \end{equation*}
    Then, since $\log(1-x)\geq -x-x^2$ for $0\leq x\leq 1$,
    \begin{align*}
        \prod_{p_k<p\leq x}\frac{1}{1-\frac{a}{p}}&=\exp\left(-\sum_{p_k<p\leq x}\log\left(1-\frac{a}{p}\right)\right)\\
        &\leq\exp\left(a\sum_{p\leq x}\frac{1}{p}\right)\exp\left(a\sum_{p\leq x}\frac{1}{p^2}\right)\\
        &=O\big((\log x)^a\big),
    \end{align*}
    as required.
\end{proof}
Combining these lemmas we are able to complete the proof of Theorem \ref{limsupthm2}.
\begin{proof}[Proof of Theorem \ref{limsupthm2}]
    By Lemmas \ref{wainexlem} and \ref{innerupperlem} with $z=x/m$,
    \begin{align}\label{waabseq}
        \left|W_a(x)\right|\leq\sum_{\substack{m\leq x\\ (m,P(a))=1}}a^{\Omega(m)}\left(\frac{a}{a+1}\left(\frac{x}{m}\right)^{v_a}\prod_{i=2}^k\frac{1}{1-\frac{a}{p_i^{v_a}}}+O\left(\left(\frac{x}{m}\right)^{r_a}\right)\right).
    \end{align}
    In the case $2<a<3$, when $r_a=0$, we have by Lemma \ref{outersumlem} and partial summation
    \begin{equation}\label{amtoaeq}
        \sum_{\substack{m\leq x\\(m,P(a))=1}}a^{\Omega(m)}=O(x(\log x)^a).
    \end{equation}
    Then when $a\geq 3$, Lemma \ref{outersumlem} gives
    \begin{equation*}
        \sum_{\substack{m\leq x\\ (m,P(a))=1}}\frac{a^{\Omega(m)}}{m^{\log_{3} a}}\leq \sum_{\substack{m\leq x\\ (m,P(a))=1}}\frac{a^{\Omega(m)}}{m}=O\big((\log x)^a\big).
    \end{equation*}
    Therefore, for all $a>2$, summing over the big $O$ term in \eqref{waabseq} yields an error term which is of an order smaller than $x^{v_a}$. To finish, we note that
    \begin{equation*}
        \sum_{\substack{m\leq x\\(m,P(a))=1}}\frac{a^{\Omega(m)}}{m^{v_a}}\leq\prod_{p_k<p\leq x}\frac{1}{1-\frac{a}{p^{v_a}}}
    \end{equation*}
    and
    \begin{align*}
        \prod_{i=2}^k\frac{1}{1-\frac{a}{p_i^{v_a}}}\prod_{p_k<p\leq x}\frac{1}{1-\frac{a}{p^{v_a}}}&\leq\prod_{p>2}\frac{1}{1-\frac{a}{p^{v_a}}}\\
        &=\prod_{p>2}\frac{1}{1-\frac{1}{p^{v_a}}}\cdot\frac{p^{v_a}-1}{p^{v_a}-a}\\
        &<\frac{3^{v_a}-1}{3^{v_a}-a}\cdot\zeta(v_a)\cdot\left(1-\frac{1}{2^{v_a}}\right)\\
        &=\frac{a-1}{a}\cdot \frac{3^{v_a}-1}{3^{v_a}-a}\cdot\zeta(v_a),
    \end{align*}
    so that the right-hand side of \eqref{waabseq} is strictly bounded above by 
    \begin{equation*}
        \left(\frac{a-1}{a+1}\cdot\frac{3^{v_a}-1}{3^{v_a}-a}\cdot\zeta(v_a)\right)x^{v_a}+o(x^{v_a}). 
    \end{equation*}
    This completes the proof.
\end{proof}
\begin{remark}
    The same argument allows us to bound the positive term sum
    \begin{equation*}
        \sum_{n\leq x}a^{\Omega(n)}
    \end{equation*}
    for all real $a>2$. More precisely, the only difference is that upon replacing each $-a$ with $a$ in Lemma \ref{wainexlem}, the bound corresponding to \eqref{ageoeq} now has a factor of $a/(a-1)$ instead of $a/(a+1)$. This results in
    \begin{equation*}
        \sum_{n\leq x}a^{\Omega(n)}\leq\left(\frac{3^{v_a}-1}{3^{v_a}-a}\cdot\zeta(v_a)\right)x^{v_a}+o(x^{v_a}).
    \end{equation*}
\end{remark}

\subsection{Proof of Theorem \ref{3thm}}

From here onwards, we let $a=3$ and focus on proving Theorem \ref{3thm}. This will be done by giving an explicit $O(x^{v_3})$ main term and $O(x(\log x)^3)$ error term for $W_3(x)$. We begin with the following lemma, which is an explicit variant of Mertens' second theorem.

\begin{lemma}\label{mertenex}
    For all $x\geq 5$,
    \begin{equation*}
        \sum_{5\leq p\leq x}\frac{1}{p}<\log\log x.
    \end{equation*}
\end{lemma}

\begin{proof}
    By the corollary of Theorem 6 in \cite{rosser1962approximate}, one has for all $x>1$,
    \begin{equation*}
        \sum_{5\leq p\leq x}\frac{1}{p}=\sum_{2\leq p\leq x}\frac{1}{p}-\frac{5}{6}\leq\log\log x + 0.27 + \frac{1}{(\log x)^2} - \frac{5}{6},
    \end{equation*}
    which is less than $\log\log x$ for all $x\geq 5$.
\end{proof}

Next, we give explicit versions of Lemma \ref{innerupperlem} and Lemma \ref{outersumlem} in the case $a=3$. Note that here we use the notation $f(x)=O^*(g(x))$ to mean $|f(x)|\leq g(x)$ for the range of $x$ in consideration.
\begin{lemma}\label{3sumlem}
    For any $z\geq 1$,
    \begin{equation}\label{3sumfull}
        \sum_{\substack{\alpha_i\geq 0\\2^{\alpha_1}3^{\alpha_2}\leq z}}(-3)^{\alpha_1+\alpha_2}=\frac{3}{4}z^{v_3}f\left(\log_2z\right)+O^*\left(z\right),
    \end{equation}
    where 
    \begin{equation}\label{fthetaeq}
        f(\theta)=\sum_{\alpha\geq 0}(-1)^{\alpha+\lfloor{\theta-v_3\alpha}\rfloor}\cdot 3^{-(v_3-1)\alpha-\left\{\theta-v_3\alpha\right\}}
    \end{equation}
    and $\{\cdot\}$ denotes the fractional part of a real number.
\end{lemma}
\begin{proof}
    Summing the geometric series over $\alpha_1$ gives
    \begin{align}
        \sum_{\substack{\alpha_i\geq 0\\2^{\alpha_1}3^{\alpha_2}\leq z}}(-3)^{\alpha_1+\alpha_2}=\sum_{3^{\alpha_2}\leq z}(-3)^{\alpha_2}\left(\frac{1}{4}-\frac{(-3)^{\lfloor\log_2(z/3^{\alpha_2})\rfloor+1}}{4}\right).\label{minasimpeq3}
    \end{align}
    Now, writing $\lfloor\log_2(z/3^{\alpha_2})\rfloor=\log_2z-\alpha_2 v_3-\{\log_2z-\alpha_2 v_3\}$ gives
    \begin{equation*}
        \sum_{3^{\alpha_2}\leq z}(-3)^{\alpha_2}\left(\frac{(-3)^{\lfloor\log_2(z/3^{\alpha_2})\rfloor+1}}{4}\right)=\frac{3}{4}z^{v_3}\widetilde{f}(\log_2z),
    \end{equation*}
    where $\widetilde{f}(\log_2(z))$ is equal to $f(\log_2(z))$ but with the sum in \eqref{fthetaeq} restricted to $\alpha\leq\log_3(z)$. We then note that
    \begin{equation*}
        z^{v_3}\left|f(\log_2z)-\widetilde{f}(\log_2z)\right|\leq z^{v_3}\sum_{\alpha>\log_3z}3^{-(v_3-1)\alpha}\leq z,
    \end{equation*}
    so \eqref{minasimpeq3} may be reduced to
    \begin{equation}\label{min3main}
        \sum_{\substack{\alpha_i\geq 0\\2^{\alpha_1}3^{\alpha_2}\leq z}}(-3)^{\alpha_1+\alpha_2}=\frac{3}{4}z^{v_3}f(\log_2z)+\frac{1}{4}\sum_{3^{\alpha_2}\leq z}(-3)^{\alpha_2} + O^*\left(\frac{3}{4}z\right).
    \end{equation}
    Finally,
    \begin{equation*}
        \left|\sum_{3^{\alpha_2}\leq z}(-3)^{\alpha_2}\right|=\left|\frac{1}{4}-\frac{(-3)^{\lfloor\log_3z\rfloor+1}}{4}\right|\leq\frac{1}{4}+\frac{3}{4}z\leq z, 
    \end{equation*}
    which when combined with \eqref{min3main}, proves the lemma.
\end{proof}
\begin{lemma}\label{132lem}
    For any $x\geq 5$,
    \begin{equation*}
        \sum_{\substack{m\leq x\\(m,6)=1}}\frac{3^{\Omega(m)}}{m}=O^*(1.32(\log x)^3).
    \end{equation*}
\end{lemma}
\begin{proof}
    Arguing as in Lemma \ref{outersumlem}, we have
    \begin{equation*}
        \sum_{\substack{m\leq x\\(m,6)=1}}\frac{3^{\Omega(m)}}{m}\leq\exp\left(3\sum_{5\leq p\leq x}\frac{1}{p}\right)\exp\left(3\sum_{5\leq p\leq x}\frac{1}{p^2}\right).
    \end{equation*}
    Here, 
    \begin{equation*}
        \exp\left(3\sum_{5\leq p\leq x}\frac{1}{p}\right)\leq(\log x)^3
    \end{equation*}
    by Lemma \ref{mertenex}, and
    \begin{equation*}
        \exp\left(3\sum_{5\leq p\leq x}\frac{1}{p^2}\right)\leq\exp\left(3\left(\sum_{p\geq 2}\frac{1}{p^2}-\frac{1}{4}-\frac{1}{9}\right)\right)\leq 1.32,
    \end{equation*}
    noting that $\sum_{p>2}1/p^2=0.452247$ (see for instance \cite{merrifield1882iii}).
\end{proof}

We now insert the above lemmas into Lemma \ref{wainexlem} to get an asymptotic formula for $W_3(x)$ with an explicit error term.
\begin{proposition}\label{w3prop}
    For all $x\geq 5$, 
    \begin{equation}\label{w3full}
        W_3(x)=\frac{3}{4}x^{v_3}\sum_{\substack{m\leq x\\(m,6)=1}}\frac{(-3)^{\Omega(m)}}{m^{v_3}}f(\log_2(x/m))+O^*\left(1.32x(\log x)^3\right),
    \end{equation}
    where $f(\theta)$ is as defined in \eqref{fthetaeq}.
\end{proposition}
\begin{proof}
    Substituting Lemma \ref{3sumlem} with $z=x/m$ into Lemma \ref{wainexlem} gives
    \begin{align*}
        W_3(x)&=\frac{3}{4}x^{v_3}\sum_{\substack{m\leq x\\(m,6)=1}}\frac{(-3)^{\Omega(m)}}{m^{v_3}}f(\log_2(x/m))+O^*\left(\sum_{\substack{m\leq x\\(m,6)=1}}\frac{(-3)^{\Omega(m)}}{m}\right).
    \end{align*}
    To complete the proof, one then uses Lemma \ref{132lem} to bound the $O^*(\cdot)$ term. 
\end{proof}
The final lemma we give is a bound for the corresponding tail of the sum appearing in \eqref{w3full}.
\begin{lemma}\label{taillem}
    For any $\varepsilon<v_3-1$,
    \begin{equation*}
        \left|\sum_{\substack{m>x\\(m,6)=1}}\frac{(-3)^{\Omega(m)}}{m^{v_3}}f(\log_2(x/m))\right|\leq\frac{3^{v_3-1}}{3^{v_3-1}-1}\cdot\frac{5^{1+\varepsilon}-1}{5^{1+\varepsilon}-3}\cdot\zeta(1+\varepsilon)\cdot x^{1+\varepsilon-v_3},
    \end{equation*}
    where $f(\theta)$ is as defined in \eqref{fthetaeq}.
\end{lemma}
\begin{proof}
    For any $\theta>0$,
    \begin{equation*}
        |f(\theta)|\leq\sum_{\alpha\geq 0} 3^{-(v-1)\alpha}=\frac{3^{v_3-1}}{3^{v_3-1}-1}.
    \end{equation*}
    Hence,
    \begin{align*}
        \left|\sum_{\substack{m>x\\(m,6)=1}}\frac{(-3)^{\Omega(m)}}{m^{v_3}}f(\log_2(x/m))\right|\leq\frac{3^{v_3-1}}{3^{v_3-1}-1}\sum_{\substack{m>x\\(m,6)=1}}\frac{3^{\Omega(m)}}{m^{v_3}}.
    \end{align*}
    To finish the proof, note that
    \begin{equation*}
        \sum_{\substack{m>x\\(m,6)=1}}\frac{3^{\Omega(m)}}{m^{v_3}}\leq\sum_{(m,6)=1}\frac{3^{\Omega(m)}}{x^{v_3-1-\varepsilon}m^{1+\varepsilon}}=x^{1+\varepsilon-v_3}\prod_{p\geq 5}\frac{1}{1-\frac{3}{p^{1+\varepsilon}}}
    \end{equation*}
    and 
    \begin{equation*}
        \prod_{p\geq 5}\frac{1}{1-\frac{3}{p^{1+\varepsilon}}}=\prod_{p\geq 5}\frac{1}{1-\frac{1}{p^{1+\varepsilon}}}\left(\frac{p^{1+\varepsilon}-1}{p^{1+\varepsilon}-3}\right)\leq\zeta(1+\varepsilon)\cdot\frac{5^{1+\varepsilon}-1}{5^{1+\varepsilon}-3}.\qedhere
    \end{equation*}
\end{proof}
Equipped with these explicit results, we can now prove Theorem \ref{3thm}. The key point here is that the periodicity of $f(\theta)$ allows us to relate the main term in \eqref{w3full} for large $x$, to an equivalent main term for some smaller $x$.
\begin{proof}[Proof of Theorem \ref{3thm}]
    To begin, note that a direct computation yields
    \begin{equation}\label{compeq}
        \max_{x\in[2^{27},2^{28}]}\left|\frac{W_3(x)}{x^{v_3}}\right|\leq 0.813+O^*(0.001).
    \end{equation}
    We were able to perform such a computation in 9 hours using a short Python script\footnote{Here we used the primeomega function from the sympy package. Note that since $W_3 (x)$ only changes at integers, and $1/x^{v_3}$ is decreasing, the maximum of $|W_3(x)/x^{v_3}|$ must occur at an integer value of $x$.} on a modern laptop with a $2.40$ GHz processor. Using \eqref{compeq}, Proposition \ref{w3prop} then gives
    \begin{align}
        &\max_{x\in[2^{27},2^{28}]}\left|\frac{3}{4}\sum_{\substack{m\leq x\\(m,6)=1}}\frac{(-3)^{\Omega(m)}}{m^{v_3}}f(\log_2(x/m))\right|\notag\\
        &\qquad\qquad\qquad\qquad\qquad=0.813+O^*\left(0.001+\max_{x\in[2^{27},2^{28}]}\left[\frac{1.32x(\log x)^3}{x^{v_3}}\right]\right)\notag\\
        &\qquad\qquad\qquad\qquad\qquad=0.813+O^*(0.154).\label{Firstostarest}
    \end{align}
    Now, combining Proposition \ref{w3prop} and Lemma \ref{taillem}, we obtain
    \begin{equation}\label{w3infty}
        \frac{W_3(x)}{x^{v_3}}=\frac{3}{4}\sum_{\substack{m\geq 1\\(m,6)=1}}\frac{(-3)^{\Omega(m)}}{m^{v_3}}f(\log_2(x/m))+o(1).
    \end{equation}
    Our goal is to then find the maximum (in absolute values) of the above main term 
    \begin{equation*}
        S_3(x):=\frac{3}{4}\sum_{\substack{m\geq 1\\(m,6)=1}}\frac{(-3)^{\Omega(m)}}{m^{v_3}}f(\log_2(x/m)).
    \end{equation*}
    By the definition \eqref{fthetaeq} of $f(\theta)$, one has $f(\theta+1)=-f(\theta)$ and consequently
    \begin{equation*}
        S(x)=(-1)^kS(x\cdot 2^k)
    \end{equation*}
    for any $k\in\Z$. Therefore, for any $x>2^{28}$ there exists $y\in [2^{27},2^{28}]$ such that $|S(x)|=|S(y)|$. So by \eqref{Firstostarest} and Lemma \ref{taillem} with $\varepsilon=0.1$,
    \begin{align*}
        \max_{x>2^{28}}|S_3(x)|&=0.813+O^*\left(0.154+\frac{3^{v_3-1}}{3^{v_3-1}-1}\cdot\frac{5^{1.1}-1}{5^{1.1}-3}\cdot\zeta(1.1)\cdot (2^{27})^{1.1-v_3}\right)\\
        &=0.813+O^*(0.158),
    \end{align*}
    which proves the theorem.
\end{proof}

\section{Further discussion and experimentation}\label{expsect}
\subsection{Possible improvements to Theorem \ref{mainthm}}
There are many avenues for lowering our value of $C=2260$ appearing in Theorem \ref{mainthm}. For example, one could extend some of our computations\footnote{See for instance the remark after Proposition \ref{m2prop}.}, or further subdivide the intervals for which we obtain bounds on $U(x)$ (see Proposition \ref{biguprop}). To limit the scope of this paper, we have chosen not to explore any of these avenues as we believe it would not lead to much further insight on the true nature of the problem. This is because vastly better bounds on $U(x)$ are required in order to prove Sun's conjecture using our framework. To get a sense of how far away we are from attaining $C=1$, suppose for argument's sake that for some $\alpha\in(1/2,1)$,
\begin{equation}\label{U09eq}
    |U(x)|\leq x^{\alpha},\qquad x\geq 2.5\cdot 10^{14}.
\end{equation}
Note that such a bound is implied by the Riemann hypothesis\footnote{The Riemann hypothesis gives $M(x)=O\big(x^{\frac{1}{2}+\varepsilon}\big)$ for any $\varepsilon>0$ (see e.g.,\ \cite[Theorem 13.24]{montgomery2007multiplicative}). With such a bound, a slight modification of the arguments in Section \ref{oddsect} also gives $U(x)=O\big(x^{\frac{1}{2}+\varepsilon}\big)$.} for sufficiently large $x$. However, even with $\alpha=0.9$, \eqref{U09eq} is far stronger than any of the bounds obtained in Proposition \ref{biguprop}.

Inserting \eqref{U09eq} into our bound \eqref{firstWexp} for $|W(x)|$ gives for all $x\geq 2.5\cdot 10^{14}$,
\begin{align}
    |W(x)|&<0.979x+\sum_{j=0}^{k-1}2^j\frac{x^{\alpha }}{2^{\alpha j}}\notag\\
    &=0.979x+x^{\alpha }\sum_{j=0}^{k-1}2^{(1-\alpha)j}\notag\\
    %&=0.979x+x^{\alpha}\frac{2^{(1-\alpha)k}-1}{2^{1-\alpha}-1}\notag\\
    &<0.979x+x\left(\frac{2^{1-\alpha}(2.5\cdot 10^{14})^{\alpha-1}}{2^{1-\alpha}-1}\right),\label{09bound}
\end{align}
where $k$ is as in \eqref{kdef}. When $\alpha=0.9$, the bound \eqref{09bound} reduces to $|W(x)|<1.53x$, which is still weaker than that conjectured by Sun. Instead, one requires $\alpha\leq 0.81$ for which \eqref{09bound} gives $|W(x)|<0.994x$. Computations suggest that the bound $|U(x)|\leq x^{0.81}$ holds for all $x\geq 8$ (see Figure \ref{fig:u09}), thereby giving support for Sun's conjecture. However, proving that such a bound holds for all $x\geq 8$ (or even $x\geq 2.5\cdot 10^{14}$) appears to be a difficult problem. Notably, even under the assumption of the Riemann hypothesis, current explicit bounds of the form $M(x)=O\left(x^{\alpha}\right)$ with $\alpha\in(1/2,1)$, require $x\geq 10^{30000}$ \cite[Theorem 2]{Simonic22}. 

\begin{figure}[!ht]
    \centering
    \includegraphics[width=1\textwidth]{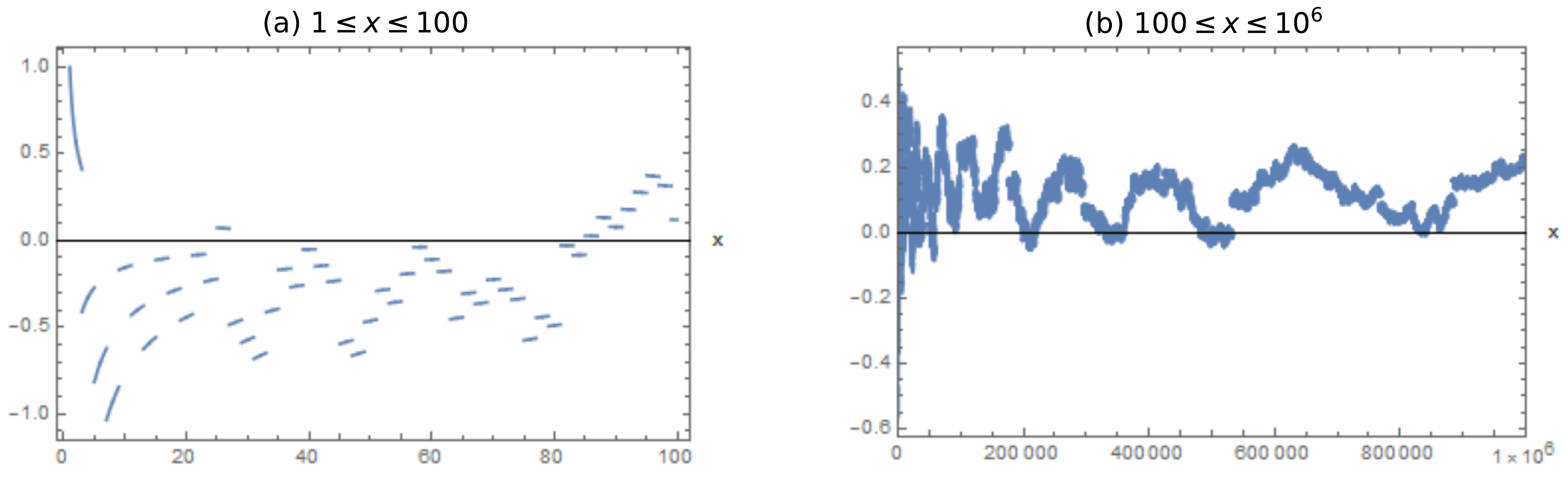}
    \caption{The function $U(x)/x^{0.81}$ for $1\leq x\leq 10^6$. Note that we have split the plot into two to show clearly the maximum attained at $x=1$, and the minimum at $x=7$.}\label{fig:u09}
\end{figure}

We also remark that independent of the strength of any potential bounds on $|U(x)|$, our method cannot be used to determine the exact value of 
\begin{equation*}
    s_W=\limsup_{x\to\infty}\frac{|W(x)|}{x}
\end{equation*}
appearing in Corollary \ref{limsupcor}; a new method would be required to understand fully the behaviour of $|W(x)|$. Recall that in our method, we first obtained bounds for $U(x)$ i.e., the sum over odd terms, and then related $U(x)$ to $W(x)$ (see Lemma \ref{wklem}). This approach was relatively elementary, since $U(x)$ can be bounded via $M(x)$, and explicit bounds for $M(x)$ have been widely studied in the literature. However, one could give a more streamlined approach that avoids using our relation between $U(x)$ and $W(x)$. In particular, one could attempt to apply directly techniques of residue calculus to a Dirichlet series related to $W(x)$, namely
\begin{equation}\label{jeq}
    J(s):=\sum_{n\geq 1}\frac{(-2)^{\Omega(n)}}{n^s}=(1+2^{1-s})^{-1}\prod_{p>2}\left(1+\frac{2}{p^s}\right)^{-1},
\end{equation}
which has singularities at $s=1+i(2k+1)\pi/\log 2$ for all integers $k$. Such an approach would likely be more involved, and increase the difficulty of getting an explicit bound on $|W(x)|$ as in Theorem \ref{mainthm}. However, it could lead to a better understanding of the value of $s_W$. 

\subsection{Computations concerning $W_a(x)$}
To simplify notation, we define
\begin{equation*}
    \widetilde{W}_a (x)=\frac{W_a (x)}{x^{\log_2a}},    
\end{equation*}
so that Conjecture \ref{newcon2} is equivalent to 
\begin{equation}\label{widetildebound}
    |\widetilde{W}_a (x)| < 1 
\end{equation}
for real $a\geq 2$ and sufficiently large $x$.

To test \eqref{widetildebound}, we plotted $\widetilde{W}_a(x)$ for a wide range of $a$ and $x\leq 10^6$, of which a few particular cases are shown in Figure \ref{fig:wa}. Note that Figure \hyperref[fig:wa]{\ref*{fig:wa}a} displays the case $a=1.5$, so as to provide a contrast to the $a\geq 2$ plots. The computational data suggests that $\widetilde{W}_a(x)$ behaves very similarly for all $a\geq 2$, and in particular seems to mimic the behaviour of the function $\widetilde{W}_2(x)=W(x)/x$. In each plot of $\widetilde{W}_a(x)$, we could clearly see the large jump discontinuities at integers divisible by %many 
powers of $2$ (e.g.,\ $x=2^m$, $x=3\cdot 2^m$, $x=5\cdot 2^m$, $\ldots$). These jump discontinuities dominated the structure of $\widetilde{W}_a(x)$, keeping the function in the range $(-1,1)$, except possibly at a few small values of $x$. On the other hand, at odd values of $x$ (corresponding to the function $U(x)$ when $a=2$), the jumps were much smaller and manifested as ``noise" in the plots.

\begin{figure}[!ht]
    \centering
    \includegraphics[width=1\textwidth]{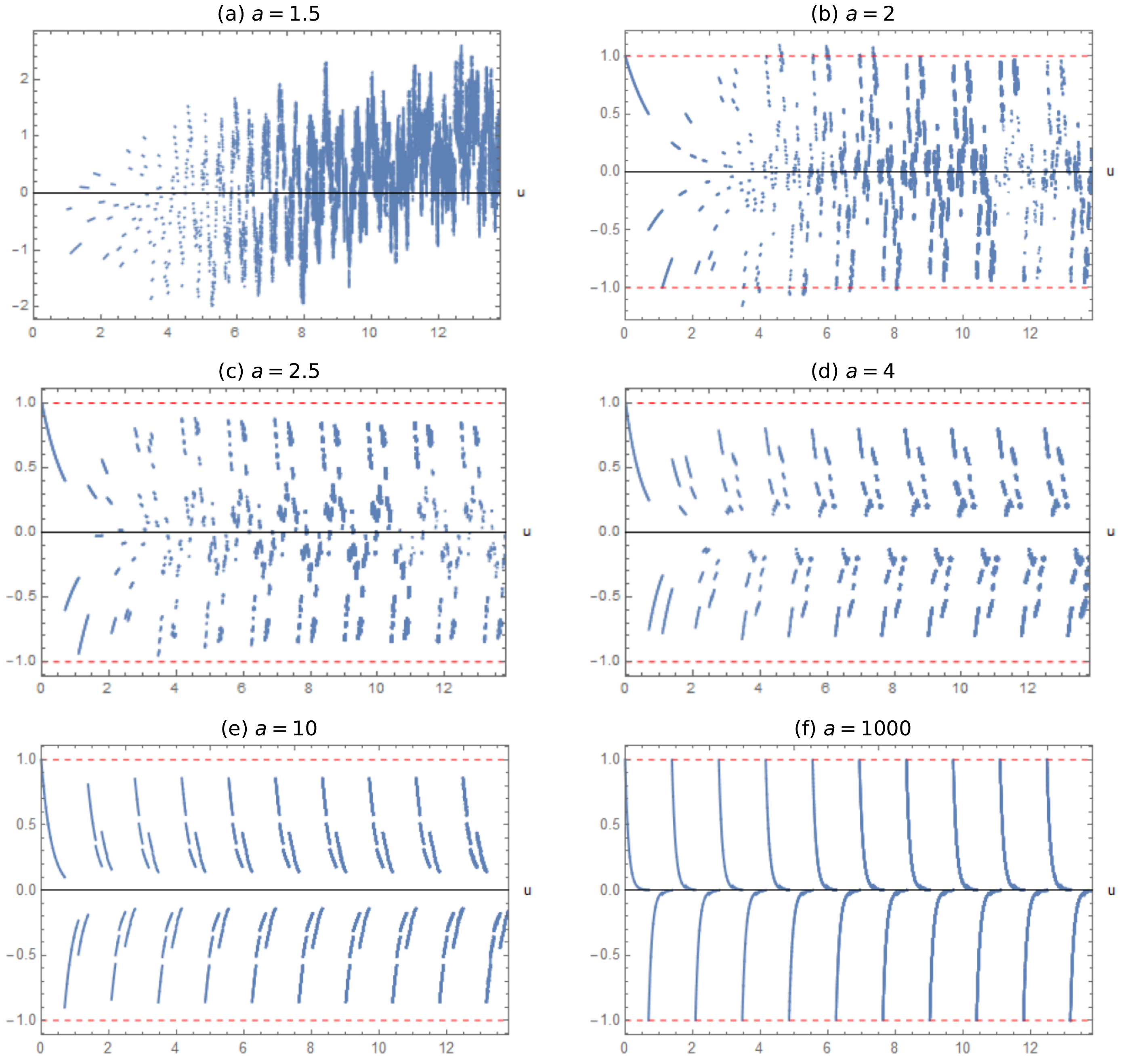}
    \caption{Plots of $\widetilde{W}_a(e^u)$ for $0\leq u\leq \log(10^6)$, for a selection of values of $a$. Note that for $a\geq 2$ and sufficiently large $x=e^u$, $\widetilde{W}_a(x)$ appears to be bounded by $[-1,1]$.}\label{fig:wa}
\end{figure}

Although our computations strongly support \eqref{widetildebound}, it seems difficult to conjecture an exact expression for $s_a=\limsup_{x\to\infty}|\widetilde{W}_a(x)|$. However, we can get a rough idea by computing the local maxima of $|\widetilde{W}_a(x)|$ over a finite range of $x$. For example, when $10^4\leq x\leq 10^6$,
\begin{align*}
    &|\widetilde{W}_2(x)|<0.9758,&&|\widetilde{W}_3(x)|<0.8106,\\
    &|\widetilde{W}_{10}(x)|<0.8581,&&|\widetilde{W}_{20}(x)|<0.9159,\\
    &|\widetilde{W}_{1000}(x)|<0.9981,&&|\widetilde{W}_{2000}(x)|<0.9991,
\end{align*}
with each of the above bounds sharp to four decimal places. In general, the main observation we made was that as $a\to\infty$, the function $|\widetilde{W}_a|$ attained local maxima closer to $1$. We thus make one final conjecture.
\begin{conjecture}
    Let $s_a$ be defined as in Theorem \ref{limsupthm2}. Then
    \begin{equation*}
        \lim_{a\to\infty}s_a=1.
    \end{equation*}
\end{conjecture}

\section*{Acknowledgements}
The authors would like to thank our supervisor Tim Trudgian, whose earlier correspondence with Mike Mossinghoff and Roger Heath-Brown inspired this paper. We are also grateful to Ofir Gorodetsky, Bryce Kerr, and G\'erald Tenenbaum for their valuable thoughts and insights on this problem.

\newpage 

\begin{section*}[A]{Appendix: Bounds on $M(x)$ and $m(x)$}
In this appendix, we obtain bounds for the Mertens function $M(x) = \sum_{n\le x} \mu(n)$ and consequently $m(x)=\sum_{n\leq x}\mu(n)/n$, where $\mu$ denotes the M\"obius function. Although such bounds exist in the literature (e.g.,\ \cite{schoenfeld69,elmarraki95,cde,bordelles2015some,Ramare13}), they are either not strong enough for our purposes, or have not been updated with recent computations and estimates. In particular, the most important type of bound we require for this work is one of the form
\begin{equation*}
    |m(x)|\leq\frac{c}{(\log x)^{2+\varepsilon}},\qquad x\geq x_0,
\end{equation*}
where $c$, $\varepsilon$ and $x_0$ are positive constants, with $c$ and $x_0$ preferably as small as possible, and $\varepsilon$ not too close to $0$. By using a standard iterative argument originally due to Schoenfeld \cite{schoenfeld69}, updated with recent estimates on the prime counting function $\psi(x)$ \cite{BKLNW_20,fiori2023cheby}, we are able to obtain the following results.

\begin{theorem}\label{thm:Mxlogxa}
    We have
    \begin{align}
        |M(x)|&\le \frac{2.91890 x}{(\log x)^2}, \hspace{0.9cm} x>1,\label{Mxbdexact1}\\
        |M(x)|&\le \frac{4.88346x}{(\log x)^{2.35}},\qquad x\geq e^{385},\label{Mxbdexact2}
    \end{align}
    and more generally,
    \begin{equation}\label{Mxbdgen}
        |M(x)|\le \frac{c_M x}{(\log x)^k},\qquad x\ge x_M,
    \end{equation}
    where permissible values of $c_M$, $k$, and $x_M$ are given in Table \ref{Mxlogxtable}.
\end{theorem}
\begin{remark}
    There has been some other recent unpublished work \cite{leeleongarxiv2024,ramare2024explicit} relating to bounds on $M(x)$. In \cite{leeleongarxiv2024}, Lee and the second author prove explicit bounds of the forms
    \begin{equation*}
        M(x)\ll x\exp\big(-\eta_1\sqrt{x}\big)\quad \text{and}\quad M(x) \ll x\exp\left(-\eta_2 (\log{x})^{\tfrac{3}{5}} (\log\log{x})^{-\tfrac{1}{5}}\right),
    \end{equation*}
    for a range of constants $\eta_i >0$. Although the bounds they obtain are asymptotically strong enough for our purposes, they only beat \eqref{Mxbdexact2} for very large $x$. Then, in \cite{ramare2024explicit}, Ramar\'e and Zuniga-Alterman obtain the bound
    \begin{equation*}
        |M(x)|\leq\frac{0.006688x}{\log x},\qquad x\geq 1\ 079\ 974,
    \end{equation*}
    which could be incorporated in our argument to give a sharper bound in \eqref{Mxbdexact1}.
\end{remark}
\begin{theorem}\label{thm:lilmx}
    We have
    \begin{align}
        |m(x)|&\le \frac{4.591}{(\log x)^2},\,\hspace{1cm} x>1,\label{lilmxb1}\\
        |m(x)|&\le \frac{7.397}{(\log x)^{2.35}},\qquad x\geq e^{390}.\label{lilmxb2}
    \end{align}
\end{theorem}
The method used to obtain Theorems \ref{thm:Mxlogxa} and \ref{thm:lilmx} is detailed in the subsequent sections \ref{psisub}--\ref{mxsub}.

We also make note of the following bounds, respectively due to Hurst \cite{hurst2018computations}, and Cohen, Dress and El Marraki \cite{cde}, which are stronger for small values of $x$.
\begin{proposition}[{\cite[Section 6]{hurst2018computations} and \cite[Theorem 5 bis]{cde}}]
    We have
    \begin{align}
        |M(x)| &\le 0.571\sqrt{x}, \qquad\, 33\le x \le 10^{16}, \label{hursts} \\
        |M(x)| &\le \frac{x}{4345}, \qquad \qquad x \ge 2160535. \label{cohendressel}
    \end{align}
\end{proposition}

\subsection{Bounds on $\psi(x)$}\label{psisub}
To begin, we first establish bounds on the function $\rho(x)=\psi(x)-[x]$, where
\begin{equation*}
    \psi(x)=\sum_{p^k\leq x}\log p
\end{equation*}
is the Chebyshev prime-counting function. To bound $\rho(x)$ we use existing bounds on the error $|\psi(x)-x|$, which is more often studied in the literature.

\begin{lemma}\label{degrader}
Suppose there exist constants $x_{\omega}$, $\omega_1$, $\omega_2$, and $\omega_3$ such that for all $x\geq x_{\omega}$,
\begin{equation}\label{psikadform}
    |\psi(x)-x|\le \omega_1 x(\log{x})^{\omega_2}\exp\left(-\omega_3\sqrt{\log{x}}\right).
\end{equation}
Let $\lambda_h= 4(\omega_2+\beta)^2 / \omega_3^2$. Then for all $\beta\ge 0$, $x \ge x_h = \max(\exp(\lambda_h),x_{\omega})$, we have
\begin{equation*}
    |\rho(x)|< \frac{B_\omega x}{(\log x)^{\beta}},
\end{equation*}
with $B_\omega = \omega_1(\log x_h)^{\omega_2 +\beta}\exp (-\omega_3 \sqrt{\log x_h}) + \frac{(\log x_h)^\beta}{x_h}$.
\end{lemma}

\begin{proof}
Let $h(x)=(\log x)^{\omega_2+\beta}\exp (-\omega_3 \sqrt{\log x})$. Since $h'(x) < 0$ whenever
\begin{equation*}
    x > \exp\left( \frac{4(\omega_2+\beta)^2}{\omega_3^2}\right) = \exp(\lambda_h),
\end{equation*}
we obtain for $x\ge x_h$,
\begin{equation*}
    |\psi(x)-x|<\frac{B_0 x}{(\log x)^\beta},
    \quad\text{where}\quad
    B_0 = \omega_1(\log x_h)^{\omega_2 +\beta}\exp (-\omega_3 \sqrt{\log x_h}).
\end{equation*}
Because $(\log x)^\beta /x$ is decreasing in $x$ and $|\rho(x)|=|\psi(x)-[x]|<|\psi(x)-x|+1$, it then follows that for $x\ge x_h$,
\begin{equation*}
    |\rho(x)|< \frac{B_0 x + ((\log x)^\beta /x)x}{(\log x)^\beta} \le \frac{B_\omega x}{(\log x)^{\beta}}, \quad\text{where}\quad B_\omega = B_0 + \frac{(\log x_h)^\beta}{x_h}. 
    \qedhere
\end{equation*}
\end{proof}

\begin{lemma}\label{degrader2}
We have the following explicit bounds for $\rho(x)$:
\begin{equation*}
    |\rho(x)|< \frac{B_\beta x}{(\log x)^{\beta}} \quad \text{for} \quad x\ge e^{20},
\end{equation*}
with
\begin{align*}
    B_1 &= 8.96237\cdot 10^{-4}, \qquad \qquad B_2 = 1.88209\cdot 10^{-2}, \\
    B_3 &= 3.95239\cdot 10^{-1}, \qquad\qquad B_4 = 7.51090\cdot 10^{1}.
\end{align*}
\end{lemma}

\begin{proof}
We will only prove the case for $\beta =1$, since the other cases are similarly proven mutatis mutandis. To begin, we note that the values in \cite[Table~8]{BKLNW_20} make the following estimate explicit:
\begin{equation}\label{eqn:BroadbentEtAlIntervals}
    |\psi(x) - x| \leq \varepsilon(b,b') x
    \quad\text{for}\quad
    e^b \leq x \leq e^{b'},
\end{equation}
by providing corresponding values of $b$, $b'$ and $\epsilon(b,b')>0$. In particular, if $b=20$ and $b'=21$, then for $e^{20} \le x \le e^{21}$,
\begin{align*}
    |\rho(x)| 
    &< |\psi(x)-x|+1 < 4.26760\cdot 10^{-5} x +1 \\
    &< \left( 4.26760\cdot 10^{-5}\log(e^{21}) +\frac{\log (e^{20})}{e^{20}} \right)\frac{x}{\log x} = \frac{8.96237\cdot 10^{-4} x}{\log x},
\end{align*}
since $(\log x)^\beta  /x$ is decreasing for $x>e^\beta$. Repeating this process for every subsequent interval in \cite[Table 8]{BKLNW_20} then gives that $B_1 = 8.96237\cdot 10^{-4}$ holds for all $e^{20} \le x \le e^{3000}$. 

To complete the proof, we use \cite[Corollary~1.3]{fiori2023cheby} instead, which states
\begin{equation}\label{psikad}
       |\psi(x)-x|\leq 9.22022\, x (\log{x})^{1.5}\exp\left(-0.8476836\sqrt{\log{x}}\right) \quad \text{for} \quad x > 2.
\end{equation}
Inserting \eqref{psikad} into into Lemma \ref{degrader} gives 
\begin{equation*}
    |\rho(x)| < \frac{3.2\cdot 10^{-11} x}{\log x}  \quad \text{for} \quad x \ge e^{3000};
\end{equation*}
the constant $3.2\cdot 10^{-11}$ is less than $B_1 = 8.96237\cdot 10^{-4}$, so we are done.
\end{proof}

\subsection{The iterative process}\label{itsub}
We now describe the main iterative process, and prove Theorem \ref{thm:Mxlogxa}. The main tools below are adapted from the work of El Marraki \cite{elmarraki95}, which improves upon an argument of Schoenfeld \cite{schoenfeld69}. In what follows, let the function $N(x)$ be defined as 
\begin{equation*}
    N(x) = \sum_{1\le n \le x} \mu(n) \log (n).
\end{equation*}

\begin{proposition}[\text{cf.\ \cite[Proposition~5]{elmarraki95}}]\label{Nx}
Suppose we have the following bounds:
\begin{equation}\label{bform}
|M(u)| \le \frac{Au}{\log^{\alpha}u} \quad\text{ for }\, u\ge u_0, \qquad |\rho(v)| \le \frac{Bv}{\log^\beta v}\quad \text{ for }\, v\ge v_0,
\end{equation}
with $\alpha,\beta, \ge 0$ and $\alpha<\beta$.
Define $y$ by
\begin{equation*}
y:= y(x) = \exp\left(D(\log x)^{\frac{\alpha +1}{\beta+1}}\right), \quad \text{where} \quad D=\left( \frac{3B\beta}{\pi^2 A}\right)^{\frac{1}{\beta +1}}.
\end{equation*}
Consider $\lambda$ such that
\begin{equation*}
\lambda\ge \left( \frac{0.8B}{AD^\beta \log D}\right)^{\frac{\beta+1}{\beta-\alpha}},
\end{equation*}
and define:
\begin{align*}
    f(q) = \frac{2AD}{(1-q)^\alpha}+\frac{6B}{\pi^2 D^\beta}(1-q), \qquad &q_0=D\lambda^{\frac{\alpha-\beta}{\beta+1}}, \\
    c_N = \max(f(0),f(q_0)), \qquad \qquad &a= \frac{\alpha\beta -1}{\beta+1}.
\end{align*}
Then, for $x\ge x_1:= e^{\lambda_1}$, $\tfrac{x_1}{y(x_1)}\ge u_0,\, y(x_1)\ge v_0,\,D>1,\, \tfrac{0.017Bx_1}{(\log y(x_1))^\beta}>\tfrac{1}{2}$ and $\big(\log \tfrac{x_1}{y(x_1)}\big)^\alpha \ge 2.5A$, where
\begin{equation}\label{x1max}
    \lambda_1 = \max\left\lbrace \lambda,\, \frac{\beta(\alpha+1)}{\beta+1},\, \left(\frac{D(\alpha+1)}{\beta+1}\right)^{\big(1-\tfrac{\alpha+1}{\beta+1}\big)^{-1}} \right\rbrace,
\end{equation}
we have
\begin{equation}\label{Nrelation}
|N(x)| < \frac{c_N x}{(\log x)^a},\qquad x\ge e^{\lambda_1}.
\end{equation}
\end{proposition}
\begin{proof}
    The proof is that of \cite[Proposition~5]{elmarraki95}. We only include some simplification of the conditions. In particular, \cite[Proposition~5]{elmarraki95} originally requires that we have $x/y \ge u_0$, $ 2(0.017B)x / (\log y)^\beta>1$, and $ (\log (x/y))^\alpha\ge 2.5A$. However, note that
    \begin{equation*}
        \frac{x}{y(x)}=\frac{x}{\exp\left(D(\log x)^{\frac{\alpha +1}{\beta+1}}\right)} \ge \frac{x_1}{y(x_1)},
    \end{equation*}
    since $x/y$ is increasing for 
    \begin{equation*}
        x \ge x_1\ge \exp\Bigg(\left(\frac{D(\alpha+1)}{\beta+1}\right)^{\big(1-\tfrac{\alpha+1}{\beta+1}\big)^{-1}}\Bigg).
    \end{equation*}
    
    Similarly, $2(0.017B)x / (\log y)^\beta>1$ is satisfied whenever $x \ge x_1\ge \exp\big(\beta(\tfrac{\alpha+1}{\beta+1})\big)$ and $2(0.017B)x_1 / (\log y(x_1))^\beta>1$ holds.
\end{proof}

\begin{proposition}\label{Mx}
    Suppose that
    \begin{equation}\label{m0_assump}
        |M(x)|\le c_0 x, \qquad x\ge x_0.
    \end{equation}
    Let $\lambda_1,\,a,\, c_N,$ be defined as in Proposition \ref{Nx}, and let $k = a+1$. Then for $x\ge x_M := \max(e^{\lambda_1}, x_0^{10/9})$,
    \begin{equation}
        |M(x)| < \frac{c_M x}{(\log x)^k},
    \end{equation}
    with
    \begin{equation*}
        c_M = c_N + \left( c_0 + \frac{c_N}{(\log x_0)^{k}}\right)\frac{x_0 {\lambda_1}^k}{e^{\lambda_1}} + \frac{c_N {\lambda_1}^k}{(\log x_0)^{k +1} e^{{\lambda_1}/10}}+\frac{c_N}{(9/10)^{k +1}{\lambda_1}}.
    \end{equation*}
\end{proposition}
\begin{proof}
    The proof is that of \cite[Proposition~7]{elmarraki95}. However, we have generalised their result by replacing their bound for $M(C)$ with \eqref{m0_assump}, in the relation (see \cite[p. 147]{elmarraki95})
    \begin{equation}\label{Meq}
        M(x) \le M(C) + \frac{|N(C)|}{\log C} + \frac{|N(x)|}{\log x}+\int_C^x \frac{|N(u)|}{u(\log u)^2}\,\mathrm{d}u,
    \end{equation}
    where $x\ge C>1$.
\end{proof}

If one has bounds for $M(x)$ and $\rho(x)$ in the appropriate form of \eqref{bform}, then using Propositions \ref{Nx} and \ref{Mx} back and forth reveals an iterative process that yields increasingly sharper bounds for $N(x)$ and $M(x)$, albeit at the cost of simultaneously increasing $x_M$.

\begin{proof}[Proof of Theorem \ref{thm:Mxlogxa}]
    To begin, we describe the iterative procedure that gives \eqref{Mxbdgen} and each entry in Table \ref{Mxlogxtable}. First, fix $\beta$, and use Lemma \ref{degrader2} to obtain a bound for $\rho(x)$. Then, use \eqref{cohendressel} as the initial starting point in \eqref{bform}, as well as to give \eqref{m0_assump} with $(c_0,x_0)=(1/4345,2160535)$. Next, apply Propositions \ref{Nx} and \ref{Mx} to obtain a bound for $M(x)$. Here, one chooses $\lambda$ to be large enough so that each of the conditions in Proposition \ref{Nx} are satisfied. Use the result as a new starting point, then apply Propositions \ref{Nx} and \ref{Mx} again. This process can be repeated for as long as the conditions of the propositions can be fulfilled. 

    Now we prove the first assertion \eqref{Mxbdexact1} of Theorem \ref{thm:Mxlogxa}. Note that the second assertion \eqref{Mxbdexact2} is just the final entry in Table \ref{Mxlogxtable} upon noting that $47/20=2.35$. We begin by using a direct computation to see that for all $1 < x\le 32$,
    \begin{equation}\label{Mbound0}
        |M(x)|\le 4 < \frac{2.16537x}{(\log x)^{2}}.
    \end{equation}
    At the same time, for $33\le x \le 10^{16}$, \eqref{hursts} implies
    \begin{equation}\label{Mbound1}
        |M(x)|\le \frac{0.571(\log x)^2}{\sqrt{x}}\frac{x}{(\log x)^2} \le \frac{1.23643x}{(\log x)^2},
    \end{equation}
    using the fact that $(\log x)^2/\sqrt{x}$ is maximised at $x=e^4$. Similarly, \eqref{cohendressel} implies that for $10^{16}\le x\le e^{93}$, 
    \begin{equation}\label{Mbound2}
        |M(x)|\le \frac{93^2}{4345}\frac{x}{(\log x)^2}\le \frac{1.99057x}{(\log x)^2},
    \end{equation}
    while for $e^{93} \le x \le e^{385}$, the second row of Table \ref{Mxlogxtable} implies
    \begin{equation}\label{Mbound3}
        |M(x)|\le (6.09073\cdot 10^{-2})(\log x)^{2-(27/20)}\frac{x}{(\log x)^2} \le \frac{2.91890x}{(\log x)^2}.
    \end{equation}
    Finally, for $x\ge e^{385}$, the last row of Table \ref{Mxlogxtable} tells us that
    \begin{equation}\label{Mbound4}
        |M(x)|\le \frac{4.88346}{(\log x)^{7/20}}\frac{x}{(\log x)^2}\le \frac{0.60788x}{(\log x)^2},
    \end{equation}
    and taking the maximum of all the above estimates \eqref{Mbound0}--\eqref{Mbound4} completes the proof.
\end{proof}
\begin{remark}
    With each iteration, there is an increase in the threshold of validity for which the estimates hold. If one is willing to accept this, then the iterative process allows us to secure an exponent $k$ that is as close to $\beta$ as we wish. This is since 
    \begin{equation*}
        k=\beta-(\beta-\alpha)\left( \frac{\beta}{\beta +1}\right)^{n_k},
    \end{equation*}
    where $n_k$ counts the number of iterations, hence $k$ approaches the limit $\beta$ as the number of iterations increases. Furthermore, Schoenfeld notes that even utilising bounds of the stronger form \eqref{psikadform} in \eqref{bform} ultimately will not yield anything better than
    \begin{equation*}
        M(x) = O\left( \frac{x}{(\log x)^k}\right),
    \end{equation*}
    with arbitrarily large $k$ (see \cite[Section 4]{schoenfeld69} for more details).   
\end{remark}

\begin{table}[!ht]
    \centering
    \footnotesize
    \begin{tabular}{c|c|c|c||c|c|c}
        $\alpha$ & $A$ & $\beta$ & $B$ & $k$ & $c_M$ & $\log x_M$\\
        \hline
        $ 0 $\, & \, $1/4345$ \, & \,$ 4 $\, & \, $7.5109\cdot 10$ \, &\, $ 4/5 $\, & \, $7.80973\cdot 10^{-3}$ \, &\, $ 43 $\\
        $ 4/5 $\, & \, $7.80973\cdot 10^{-3}$ \, & \,$ 3 $\, & \, $3.95239\cdot 10^{-1}$ \, &\, $ 27/20 $\, & \, $6.09073\cdot 10^{-2}$ \, &\, $ 93 $\\
        \hline
        $ 0 $\, & \, $1/4345$ \, & \,$ 2 $\, & \, $1.88209\cdot 10^{-2}$ \, &\, $ 2/3 $\, & \, $2.55758\cdot 10^{-3}$ \, &\, $ 161 $\\
        $ 2/3 $\, & \, $2.55758\cdot 10^{-3}$ \, &  \,$ 2 $\, & \, $1.88209\cdot 10^{-2}$ \, &\, $ 10/9 $\, & \, $1.30895\cdot 10^{-2}$ \, &\, $ 192 $\\
        $10/9 $\, & \, $1.30895\cdot 10^{-2}$ \, &  \,$ 3 $\, & \, $3.95239\cdot 10^{-1}$ \, &\, $ 19/12 $\, & \, $9.12303\cdot 10^{-2}$ \, &\, $ 233 $\\
        $ 19/12 $\, & \, $9.12303\cdot 10^{-2}$ \, &  \,$ 3 $\, & \, $3.95239\cdot 10^{-1}$ \, &\, $ 31/16 $\, & \, $4.30429\cdot 10^{-1}$ \, &\, $ 288 $\\
        $ 31/16 $\, & \,  $4.30429\cdot 10^{-1}$ \, &  \,$ 4 $\, & \, $7.5109\cdot 10$ \, &\, $47/20 $\, & \, $4.88346$ \, &\, $ 385 $
    \end{tabular}
    \caption{Values of $k$ and $c_M$ for $x \ge x_M$. Inputs $A$, $\alpha$, $B$, $\beta$, are defined as in Proposition \ref{Nx} and outputs $c_M$, $k$, $x_M$, are defined as in Proposition \ref{Mx}. The first two rows are the result of an iterative process we performed starting with $(\alpha,\,\beta)=(0,\,4)$. The last five rows are a separate procedure we ran starting with $(\alpha,\,\beta)=(0,\,2)$.}
    \label{Mxlogxtable}
\end{table}

\subsection{Bounds on $m(x)$}\label{mxsub}
    Finally we show how to convert our bounds on $M(x)$ to bounds on $m(x)$ and thereby prove Theorem \ref{thm:lilmx}. This will be done using a result of Balazard \cite{balazard2012elementary}.
    \begin{proof}[Proof of Theorem \ref{thm:lilmx}]
        First we prove \eqref{lilmxb1} by considering different ranges of $x$. To begin, for $1<x\leq 33$, we have by a simple computation
        \begin{equation}\label{m2bound1}
            |m(x)|\leq\frac{0.8}{(\log x)^2}.
        \end{equation}
        Next we let $33<x\leq 10^{16}$. By \cite[Equation (8)]{balazard2012elementary},
        \begin{equation}\label{bordbound}
            |m(x)|\leq\frac{|M(x)|}{x}+\frac{1}{x^2}\int_1^x|M(t)|\mathrm{d}t+\frac{8}{3x}.
        \end{equation}
        Hence, using \eqref{hursts},
        \begin{equation}\label{m2bound2}
            |m(x)|\leq\frac{0.571}{\sqrt{x}}+\frac{0.571}{x^2}\left[\frac{2x^{3/2}}{3}-1\right]+\frac{8}{3x}\leq\frac{3.01}{(\log x)^2}.
        \end{equation}
        Finally, we consider $x>10^{16}$. Again, we use \eqref{bordbound} but split the integral into three. Namely, using \eqref{hursts} and \eqref{Mxbdexact1},
        \begin{align}\label{splitinteq}
\int_1^x|M(t)|\mathrm{d}t&=\int_1^{33}|M(t)|\mathrm{d}t+\int_{33}^{10^{16}}|M(t)|\mathrm{d}t+\int_{10^{16}}^x|M(t)|\mathrm{d}t\notag\\
            &\leq 59+0.571\int_{33}^{10^{16}}t^{1/2}\mathrm{d}t+2.9189\int_{10^{16}}^x\frac{t}{(\log t)^2}\mathrm{d}t\notag\\
            &\leq 3.80667\cdot 10^{23}+\frac{1.67204x^2}{(\log x)^2},
        \end{align}
        where in the last line we used \cite[Lemma 7]{bordelles2015some} to bound  the integral with $t/(\log t)^2$. Substituting \eqref{splitinteq} into \eqref{bordbound} then gives
        \begin{align}\label{m2bound3}
            |m(x)|\leq\frac{2.9189}{(\log x)^2}+\frac{1.67204}{(\log x)^2}+\frac{3.80667\cdot 10^{23}}{x^2}+\frac{8}{3x}\leq\frac{4.591}{(\log x)^2}.
        \end{align}
        Combining \eqref{m2bound1}, \eqref{m2bound2} and \eqref{m2bound3}, we obtain the desired result.

        To prove \eqref{lilmxb2}, we similarly use \eqref{bordbound}. So, let $x\geq e^{390}$. In this case, \eqref{Mxbdexact1} and \eqref{Mxbdexact2} give
        \begin{align}\label{splitinteq2}
            \int_1^x|M(t)|\mathrm{d}t&=\int_1^2|M(t)|\mathrm{d}t+\int_2^{ e^{385}}|M(t)|\mathrm{d}t+\int_{ e^{385}}^x|M(t)|\mathrm{d}t\notag\\
            &\leq 1+2.9189\int_2^{ e^{385}}\frac{t}{(\log t)^2}\mathrm{d}t+4.88346\int_{ e^{385}}^x\frac{t}{(\log t)^{2.35}}\mathrm{d}t\notag\\
            &\leq 2.5186\cdot 10^{329}+\frac{2.5122x^2}{(\log x)^{2.35}},
        \end{align}
        where in moving to the last line we again used \cite[Lemma 7]{bordelles2015some} to bound the integral with $t/(\log t)^{2.35}$. Substituting \eqref{splitinteq2} into \eqref{bordbound} then yields
        \begin{equation*}
            |m(x)|\leq\frac{4.88346}{(\log x)^{2.35}}+\frac{2.5122}{(\log x)^{2.35}}+\frac{2.5186\cdot 10^{329}}{x^2}+\frac{8}{3x}\leq\frac{7.397}{(\log x)^{2.35}},
        \end{equation*}
        as desired.
    \end{proof}
\end{section*}

\newpage

\printbibliography
\end{document}